\documentclass[11pt]{amsart}

\usepackage[T1]{fontenc}
\usepackage{lmodern}
\usepackage{amsmath,amssymb,amsthm,mathtools}
\usepackage{microtype}
\usepackage{enumitem}
\usepackage[hidelinks]{hyperref}
\usepackage[a4paper,margin=1.15in]{geometry}
\hypersetup{
  pdftitle={Transversal Holder Criteria and Dini-Zygmund Endpoint Regularity for Hyperbolic Harmonic Mappings},
  pdfauthor={Hong-Ping Li and Suling Tan},
  pdfsubject={Hyperbolic harmonic mappings and endpoint regularity}
}

\makeatletter
\@namedef{subjclassname@2020}{\textup{2020} Mathematics Subject Classification}

\makeatother

\numberwithin{equation}{section}

\newtheorem{theorem}{Theorem}[section]
\newtheorem{proposition}[theorem]{Proposition}
\newtheorem{lemma}[theorem]{Lemma}
\newtheorem{corollary}[theorem]{Corollary}

\theoremstyle{definition}

\newcommand{\R}{\mathbb R}
\newcommand{\Hh}{\mathbb H}

\newcommand{\Pcal}{\mathcal P}
\newcommand{\Vcal}{\mathcal V}
\newcommand{\Mcal}{\mathcal M}
\newcommand{\Dcal}{\mathcal D}
\newcommand{\Lip}{\operatorname{Lip}}
\newcommand{\supp}{\operatorname{supp}}
\newcommand{\dd}{\,\mathrm{d}}
\newcommand{\norm}[1]{\left\lVert #1\right\rVert}

\title[Transversal regularity for hyperbolic harmonic mappings]
{Transversal H\"older Criteria and Dini--Zygmund Endpoint Regularity for Hyperbolic Harmonic Mappings}

\author{Suling Tan}
\address{School of Mathematical Sciences, Huaqiao University, Xiamen 361021, China}
\email{sltan@hqu.edu.cn}

\author{Hong-Ping Li}
\address{School of Mathematical Sciences, Huaqiao University, Quanzhou 362021, China}
\email{lhp306@hqu.edu.cn}



\begin{document}

\begin{abstract}
Let $n\ge3$, and let $u$ be a bounded hyperbolic harmonic mapping on
the real hyperbolic upper half-space $\mathbb H^n$. First, for
$0<\alpha<1$, we establish quantitative transversal and modulus
criteria for the global $\alpha$-H\"older continuity of $u$. Then,
we construct lacunary examples to show that these criteria are invalid when
$\alpha=1$. Finally, under additional integral conditions, we obtain the corresponding global Lipschitz continuity and give Dini--Zygmund conditions on the boundary trace that ensure the Lipschitz continuity.
\end{abstract}

\maketitle

\section{Introduction}
Let $n\ge3$ and $d=n-1$. For $0<\alpha\le1$ and a subset
$E$ of the Euclidean space $\R^d$, we write $\Lambda_\alpha(E)$
for the class of all functions $f$ which satisfy that there is a constant
$C>0$ such that
\[
  |f(x)-f(y)|\le C|x-y|^\alpha,
  \qquad x,y\in E.
\]
Obviously, when $0<\alpha<1$, $\Lambda_\alpha(E)$ is the usual H\"older class, while
$\Lambda_1(E)$ is the Lipschitz class.

The modern modulus approach begins with Dyakonov's work on analytic
functions in the unit disk $\mathbb D$. Let $f$ be analytic in
$\mathbb D$ and continuous on $\overline{\mathbb D}$, the closure of
$\mathbb D$. He discussed the relations among the following four quantities:
\begin{align*}
 N_1(f)&=\sup_{z\ne w\in\overline{\mathbb D}}
 \frac{|f(z)-f(w)|}{|z-w|^\alpha},\\
 N_2(f)&=\sup_{z\ne w\in\overline{\mathbb D}}
 \frac{\bigl||f(z)|-|f(w)|\bigr|}{|z-w|^\alpha},\\
 N_3(f)&=\sup_{\zeta\ne\eta\in\partial\mathbb D}
 \frac{\bigl||f(\zeta)|-|f(\eta)|\bigr|}{|\zeta-\eta|^\alpha}
 +\sup_{z\in\mathbb D}
 \frac{P[|f|](z)-|f(z)|}{(1-|z|)^\alpha}
 \end{align*}
 and
 \begin{align*}
 N_4(f)&=\sup_{\zeta\ne\eta\in\partial\mathbb D}
 \frac{\bigl||f(\zeta)|-|f(\eta)|\bigr|}{|\zeta-\eta|^\alpha}
 +\sup_{0<r<1,\ \zeta\in\partial\mathbb D}
 \frac{\bigl||f(r\zeta)|-|f(\zeta)|\bigr|}{(1-r)^\alpha}.
\end{align*}
Here, $P[|f|]$ denotes the Poisson extension of the boundary function
$|f|$. For $0<\alpha<1$, Dyakonov proved that these four quantities
are equivalent
(see \cite[Theorem~2]{Dyakonov1997}). Shorter proofs of the
equivalences involving $N_2(f)$ and $N_3(f)$ were later given by
Pavlovi\'c
(see \cite[Theorems~A and B]{Pavlovic1999}).

For an integer $N\ge2$, let $\mathbb B^N$ denote the open unit ball
in $N$-dimensional Euclidean space $\mathbb R^N$, and let
$\overline{\mathbb B^N}$ denote its closure. Pavlovi\'c obtained a real harmonic counterpart on
$\mathbb B^N$ (see \cite[Theorem~1]{Pavlovic2007}). More precisely, for $0<\alpha<1$, if
$U:\overline{\mathbb B^N}\to\mathbb R$ is continuous on
$\overline{\mathbb B^N}$ and harmonic in $\mathbb B^N$, then the global $\alpha$-H\"older continuity of $U$ is
equivalent to the boundary $\alpha$-H\"older continuity of $|U|$, and also to the
uniform radial condition
\[
  \bigl||U(r\zeta)|-|U(\zeta)|\bigr|
  \le C(1-r)^\alpha.
\]

Ravisankar later considered bounded smooth domains and proved that
the H\"older regularity along a family of curves transverse to the
boundary implies the global H\"older regularity
(see \cite{Ravisankar2013}).

In the setting of the Euclidean upper half-space $\R^N_+$, where
$N\ge2$, Markovi\'c showed that, for $0<\alpha<1$, the global
seminorm, the boundary seminorm, and the seminorm restricted to
vertical lines are quantitatively equivalent for bounded
vector-valued harmonic mappings
(see \cite[Theorem~1.5]{Markovic2025half}). For real-valued harmonic
functions, the corresponding modulus characterization was given in
\cite[Theorem~1.2]{Markovic2025half}. See also
\cite{Markovic2025epi} for a transversal result on
strict epigraphs of Lipschitz functions.

Let $n\ge3$ and $d=n-1$. We use the upper-half-space model
\[
  \Hh^n=\R^d\times(0,\infty)
\]
of real hyperbolic space. The hyperbolic Laplacian on $\Hh^n$ is
given by
\[
  \Delta_hu(x,y)
  =y^2\bigl(\Delta_xu(x,y)+\partial_y^2u(x,y)\bigr)
   -(d-1)y\partial_yu(x,y),
\]
where $x\in\R^d$, $y>0$, and $\partial_yu$ denotes the derivative
of $u$ with respect to the vertical variable $y$ (see
\cite[Chapter~3]{Stoll2016}). For a positive integer $m$, a mapping
$u:\Hh^n\to\R^m$ is called hyperbolic harmonic
if each of its components satisfies $\Delta_hu=0$.

The implication from the regularity of a boundary datum to the
regularity of its hyperbolic Poisson extension has been studied by
many authors. For example, boundary values and function spaces for
hyperbolic harmonic functions were investigated by Grellier and
Jaming (see \cite{GrellierJaming1999,Jaming1999}). In particular,
for $0<\alpha<1$, Grellier and Jaming proved that the hyperbolic
Poisson extension of an $\alpha$-H\"older continuous boundary
function is $\alpha$-H\"older continuous on the whole ball
(see \cite[Theorem~5.3]{GrellierJaming1999}). In
\cite{ChenEtAl2018}, Chen et al.\ considered the nonhomogeneous
equation $\Delta_hu=\psi$ on the hyperbolic ball. Under a weighted
integrability assumption, they obtained the representation
\[
  u=P_h[\varphi]-G_h[\psi]
\]
(see \cite[Theorem~1.1]{ChenEtAl2018}). Then they proved that the
hyperbolic Poisson extension preserves the modulus associated with
a fast majorant (see \cite[Theorem~1.3]{ChenEtAl2018}), and the
Green potential is Lipschitz under the condition
$|\psi(x)|\le M(1-|x|^2)$
(see \cite[Theorem~1.4]{ChenEtAl2018}). In particular, when
$n\ge3$, Lipschitz boundary data give rise to globally Lipschitz
hyperbolic Poisson extensions
(see \cite[Theorem~1.3]{ChenEtAl2018}). For related results on
measure data, local H\"older regularity, and smoothness up to the
boundary, see
\cite[Theorems~2.1, 3.1, and 5.5]{MateljevicMutavdzic2022}.

A natural problem is whether the corresponding opposite implication
holds in the hyperbolic setting. More precisely, let
$0<\alpha<1$ and $m\ge1$, and suppose that
$u:\Hh^n\to\R^m$ is a bounded hyperbolic harmonic mapping. Assume
that there is a constant $C>0$ such that
\[
  |u(x,s)-u(x,t)|
  \le C|s-t|^\alpha,
  \qquad x\in\R^d,\quad s,t>0.
\]
\noindent\textbf{Problem 1.1.}
Does $u$ have a globally $\alpha$-H\"older continuous extension to
$\overline{\Hh^n}$?

It is known that in the Euclidean setting, the discussions of the
corresponding implication are based on the fact that every partial
derivative of a harmonic function is again harmonic
(see \cite[Theorem~1.5]{Markovic2025half}). But, in general, this property fails for
the hyperbolic Laplacian. Indeed, if $\Delta_hu=0$, then
\[
\begin{aligned}
  0
  =\partial_y(\Delta_hu)
  &=\Delta_h(\partial_yu)
    +2y\bigl(\Delta_xu+\partial_y^2u\bigr)
    -(d-1)\partial_yu \\
  &=\Delta_h(\partial_yu)+(d-1)\partial_yu,
\end{aligned}
\]
where the last equality follows from $\Delta_hu=0$. Consequently,
\[
  \Delta_h(\partial_yu)=-(d-1)\partial_yu.
\]
Since $d=n-1\ge2$, it follows that $\partial_yu$ is not hyperbolic
harmonic unless $\partial_yu\equiv0$.

The purpose of the paper is to discuss Problem 1.1 by using the
hyperbolic Poisson kernel as an approximation family.
To state our main results, several notations are needed. Throughout the following definitions, let $0<\alpha\le1$.
For a mapping
$u:\Hh^n\to\R^m$, define
\[
  \Vcal_\alpha(u)
  =
  \sup_{\substack{x\in\R^d\\ s,t>0,\ s\ne t}}
  \frac{|u(x,s)-u(x,t)|}{|s-t|^\alpha}.
\]

We say that a bounded hyperbolic harmonic mapping
$u:\Hh^n\to\R^m$ has a uniform boundary trace
$f:\R^d\to\R^m$ if $f$ is bounded and continuous and
\[
  \lim_{y\to0^+}
  \norm{u(\cdot,y)-f}_{L^\infty(\R^d;\R^m)}
  =0.
\]

If $u$ extends continuously to $\overline{\Hh^n}$, we call
$f(x)=u(x,0)$ its boundary trace and put
\[
  [u]_{\alpha,\Hh^n}
  =
  \sup_{\substack{X,Y\in\overline{\Hh^n}\\ X\ne Y}}
  \frac{|u(X)-u(Y)|}{|X-Y|^\alpha}.
\]
We denote by $\Lip(\overline{\Hh^n};\R^m)$ the class of all continuous
mappings $u:\overline{\Hh^n}\to\R^m$ such that
$[u]_{1,\Hh^n}<\infty$, and write
\[
  [u]_{\Lip(\overline{\Hh^n})}
  :=
  [u]_{1,\Hh^n}.
\]
When $m=1$, we simply write $\Lip(\overline{\Hh^n})$.

For a real-valued mapping $u$ with the boundary trace
$f(x)=u(x,0)$, put
\[
  \Mcal_\alpha(u)
  =
  \sup_{\substack{x\in\R^d\\ y>0}}
  \frac{\bigl||u(x,y)|-|f(x)|\bigr|}{y^\alpha}.
\]

Our first theorem is the following hyperbolic counterpart of
Markovi\'c's transversal criterion
(cf.~\cite[Theorem~1.5]{Markovic2025half}).

\begin{theorem}
\label{thm:A}
Let $n\ge3$, $0<\alpha<1$ and $m$ be a positive integer.
Suppose that $u:\Hh^n\to\R^m$ is a bounded hyperbolic harmonic
mapping. If $\Vcal_\alpha(u)<\infty$, then $u$ extends continuously
to $\overline{\Hh^n}$ and
\[
  \Vcal_\alpha(u)
  \le [u]_{\alpha,\Hh^n}
  \le C_{d,\alpha}\Vcal_\alpha(u).
\]

In particular, if $m=1$ and $u$ extends continuously to
$\overline{\Hh^n}$ with boundary trace $f(x)=u(x,0)$, then
\[
  \Mcal_\alpha(u)
  \le [u]_{\alpha,\Hh^n}
  \le C'_{d,\alpha}\Mcal_\alpha(u).
\]
The constants $C_{d,\alpha}$ and $C'_{d,\alpha}$ depend only on
$d$ and $\alpha$.
\end{theorem}
The following proposition shows the best possibility of the restriction
$\alpha<1$ in Theorem \ref{thm:A}.
\begin{proposition}
\label{prop:counterexample}
For every $n\ge3$, the following statements hold:
\begin{enumerate}[
  label=\textup{(\arabic*)},
  ref=\textup{(\arabic*)}
]
\item\label{item:counterexample-transversal}
There exist a bounded real-valued hyperbolic harmonic function
$U$ on $\Hh^n$, continuous on $\overline{\Hh^n}$, and a constant
$C_d>0$ such that
\[
  |U(x,s)-U(x,t)|
  \le C_d|s-t|,
  \qquad x\in\R^d,\quad s,t\geq0.
\]
However, $U$ is not Lipschitz continuous on
$\overline{\Hh^n}$.

\item\label{item:counterexample-modulus}
There exist a bounded positive hyperbolic harmonic function
$V$ on $\Hh^n$, continuous on $\overline{\Hh^n}$, and a constant
$C'_d>0$ such that
\[
  \bigl||V(x,y)|-|V(x,0)|\bigr|
  \le C'_dy,
  \qquad x\in\R^d,\quad y>0.
\]
However, $V$ is not Lipschitz continuous on
$\overline{\Hh^n}$.
\end{enumerate}
\end{proposition}
Also, Proposition~\ref{prop:counterexample} shows that in general
Theorem~\ref{thm:A} is invalid when
$\alpha=1$. In the following, we study Problem 1.1 further with certain additional conditions for the case when $\alpha=1$. To formulate
our results, we introduce the following notations. For a bounded continuous mapping
$g:\R^d\to\R^m$, define
\[
  \omega_2(g,t)_\infty
  =
  \sup_{x\in\R^d}\sup_{|h|\le t}
  |g(x+h)-2g(x)+g(x-h)|,
  \qquad t>0.
\]
We also introduce the Dini--Zygmund functional
\[
  \Dcal(g)
  =
  \int_0^1
  \frac{\omega_2(g,t)_\infty}{t^2}\dd t.
\]
It follows from Lemma~\ref{lem:DZLip} below that the condition
$\Dcal(g)<\infty$ implies that $g$ is Lipschitz continuous.
For a bounded mapping $u:\Hh^n\to\R^m$ with a bounded
continuous boundary trace $f$, set
\[
  E_u(t)
  =
  \norm{u(\cdot,t)-f}_{L^\infty},
  \qquad t>0.
\]
If $u$ and $f$ are real-valued, set
\[
  E_{|u|}^{*}(t)
  =
  \sup_{0<s\le t}
  \bigl\||u(\cdot,s)|-|f|\bigr\|_{L^\infty},
  \qquad t>0.
\]

Now, we are ready to state our results.
\begin{theorem}
\label{thm:endpoint}
Let $n\ge3$, $m$ be a positive integer, and let
$u:\Hh^n\to\R^m$ be a bounded hyperbolic harmonic mapping.

\begin{enumerate}[
  label=\textup{(\arabic*)},
  ref=\textup{(\arabic*)}
]
\item\label{item:endpoint-transversal}
If $\Vcal_1(u)<\infty$, then $u$ has a unique uniform boundary
trace $f:\R^d\to\R^m$. If, in addition,
\begin{equation}\label{eq:intrinsicA}
  \int_0^1\frac{E_u(t)}{t^2}\dd t<\infty,
\end{equation}
then $u\in\Lip(\overline{\Hh^n};\R^m)$ and
\[
  [u]_{\Lip(\overline{\Hh^n})}
  \le
  C_d\left(
    \norm{f}_{L^\infty}
    +
    \int_0^1\frac{E_u(t)}{t^2}\dd t
  \right).
\]

\item\label{item:endpoint-modulus}
Suppose that $m=1$ and that $u$ extends continuously to
$\overline{\Hh^n}$ with boundary trace $f(x)=u(x,0)$. If
\begin{equation}\label{eq:intrinsicB}
  \int_0^1\frac{E_{|u|}^{*}(t)}{t^2}\dd t<\infty,
\end{equation}
then $u\in\Lip(\overline{\Hh^n})$ and
\[
  [u]_{\Lip(\overline{\Hh^n})}
  \le
  C'_d\left(
    \norm{f}_{L^\infty}
    +
    \int_0^1\frac{E_{|u|}^{*}(t)}{t^2}\dd t
  \right).
\]
\end{enumerate}
\end{theorem}

The following corollary gives the corresponding boundary
formulations at the endpoint $\alpha=1$.
\begin{corollary}
\label{cor:boundaryDZ}
Let $n\ge3$, $m$ be a positive integer, and let
$u:\Hh^n\to\R^m$ be a bounded hyperbolic harmonic mapping with
uniform boundary trace $f:\R^d\to\R^m$.

\begin{enumerate}[
  label=\textup{(\roman*)},
  ref=\textup{(\roman*)}
]
\item\label{item:boundary-DZ}
If $\Dcal(f)<\infty$, then
$
  u\in\Lip(\overline{\Hh^n};\R^m).
$

\item\label{item:boundary-modulus}
If $m=1$ and $\Dcal(|f|)<\infty$, then
$
  u\in\Lip(\overline{\Hh^n}).
$
\end{enumerate}
\end{corollary}

The remainder of the paper is organized as follows. In Section~2, we
recall the hyperbolic Poisson representation and establish the
multiplier estimates. We then prove an inverse
approximation theorem and a boundary-to-interior regularity estimate.
The last part of Section~2 is devoted to the endpoint case,
including the relation among the Dini--Zygmund condition, the inhomogeneous Besov space $B_{\infty,1}^{1}$
(defined in Section~2.3), and the approximation
by hyperbolic Poisson extensions. In Section~3, we prove the
transversal and modulus criteria, construct endpoint
counterexamples, and establish the Dini--Zygmund endpoint results.

\section{Preliminaries}

\subsection{The hyperbolic Poisson kernel and its Fourier--Bessel
multiplier}

We begin by recalling the hyperbolic Poisson kernel in the
upper-half-space model. Put
\[
  c_d=\frac{\Gamma(d)}{\pi^{d/2}\Gamma(d/2)}
\]
and, for $y>0$, define
\begin{equation}\label{eq:kernel}
  \Pcal_y(x)
  =
  c_d\frac{y^d}{\bigl(|x|^2+y^2\bigr)^d},
  \qquad x\in\R^d.
\end{equation}
The choice of the constant $c_d$ ensures that
\begin{equation}\label{eq:poisson-normalization}
\int_{\R^d}\mathcal P_y(x)\,dx=1.
\end{equation}
Moreover, \eqref{eq:kernel} gives the dilation identity
\begin{equation}\label{eq:poisson-dilation}
\mathcal P_y(x)=y^{-d}\mathcal P_1(x/y).
\end{equation}

For $g\in L^1(\R^d)$, we define its Fourier transform by
\[
  \widehat g(\xi)
  =\int_{\R^d}e^{-ix\cdot\xi}g(x)\dd x,
  \qquad \xi\in\R^d,
\]
where $x\cdot\xi=\sum_{k=1}^{d}x_k\xi_k$ denotes the
Euclidean inner product. For $h\in L^1(\R^d)$, we denote
its inverse Fourier transform by
\[
  \check h(x)
  =\frac{1}{(2\pi)^d}
   \int_{\R^d}e^{ix\cdot\xi}h(\xi)\dd\xi,
  \qquad x\in\R^d.
\]
Both transforms are extended to tempered distributions
by duality (see \cite[Definition~2.3.7]{Grafakos}).
Every bounded function is regarded as a tempered
distribution, and its Fourier transform is understood
in this sense.
Using the standard integral representation of the modified Bessel
function $K_\nu$ (see \cite[Section~6.22(5)]{Watson}), we obtain
\begin{equation}\label{eq:multiplier}
  \widehat{\Pcal_y}(\xi)
  =
  m_d(y|\xi|),
  \qquad y>0,\quad \xi\in\R^d,
\end{equation}
where $m_d(0)=1$ and
\[
  m_d(t)
  =
  \frac{2^{1-d/2}}{\Gamma(d/2)}
  t^{d/2}K_{d/2}(t),
  \qquad t>0.
\]
We next record the estimates for $m_d$ that will be used below.

\begin{lemma}\label{lem:Bessel}
Let $d\geq2$. Then the following statements hold.

\begin{enumerate}[
  label=\textup{(\roman*)},
  ref=\textup{(\roman*)}
]

\item\label{item:Bessel-positive}
The function $m_d$ is positive and strictly decreasing
on $(0,\infty)$, with
\[
  m_d(0)=1
  \qquad\text{and}\qquad
  \lim_{t\to\infty}m_d(t)=0.
\]

\item\label{item:Bessel-pointwise}
There are constants $C_d>0$ and $C_d'>0$ such that
\[
 |m_d'(t)|\leq C_d
 \begin{cases}
 t(1+|\log t|),&0<t\leq1,\ d=2,\\
 t,&0<t\leq1,\ d>2,\\
 t^{d/2-1/2}e^{-t},&t\geq1,
 \end{cases}
\]
and
\[
 0\leq1-m_d(t)\leq C_d'
 \begin{cases}
 t^2(1+|\log t|),&0<t\leq1,\ d=2,\\
 t^2,&0<t\leq1,\ d>2,\\
 1,&t\geq1.
 \end{cases}
\]

\item\label{item:Bessel-dyadic}
There are constants $C_d''>0$ and $C_d'''>0$ such that
\begin{equation}\label{eq:dyadicBessel}
 \sup_{r>0}\sum_{j=1}^{\infty}|m_d'(2^jr)|
 \leq C_d''
\end{equation}
and
\begin{equation}\label{eq:dyadicApprox}
 \sum_{j=1}^{\infty}2^{-j}|1-m_d(2^jr)|
 \leq C_d'''r,
 \qquad r>0.
\end{equation}

\end{enumerate}

All the constants above depend only on $d$.
\end{lemma}
\begin{proof}
Set $\nu=d/2$. By the standard identities and integral
representation for modified Bessel functions (see \cite[Sections
3.71(3) and 6.22(5)]{Watson}),
\[
\frac{d}{dt}\bigl(t^\nu K_\nu(t)\bigr)
   =-t^\nu K_{\nu-1}(t),
\qquad
K_\nu(t)>0,\quad K_{\nu-1}(t)>0,
\qquad t>0.
\]
Since
\[
m_d(t)=\frac{2^{1-\nu}}{\Gamma(\nu)}t^\nu K_\nu(t),
\]
we have $m_d(t)>0$ and $m_d'(t)<0$ for $t>0$. Thus $m_d$ is
positive and strictly decreasing on $(0,\infty)$.
Moreover, the asymptotic formula
\[
K_\nu(t)\sim\sqrt{\frac{\pi}{2t}}e^{-t}
\qquad\text{as }t\to\infty
\]
(see \cite[Section 7.23(1)]{Watson}) yields
\[
m_d(t)\sim
\frac{2^{1-\nu}}{\Gamma(\nu)}
\sqrt{\frac{\pi}{2}}\,
t^{\nu-1/2}e^{-t}
\qquad\text{as }t\to\infty.
\]
Since
\[
\lim_{t\to\infty}t^{\nu-1/2}e^{-t}=0,
\]
we obtain
\[
\lim_{t\to\infty}m_d(t)=0.
\]
Together with $m_d(0)=1$, this proves~\ref{item:Bessel-positive}.
Moreover, the standard estimates for modified Bessel functions give
\[
 K_{\nu-1}(t)\leq A_d
 \begin{cases}
  1+|\log t|,&0<t\leq1,\ d=2,\\
  t^{1-\nu},&0<t\leq1,\ d>2,\\
  t^{-1/2}e^{-t},&t\geq1,
 \end{cases}
\]
where $A_d>0$ is a constant depending only on $d$
(see \cite[Sections~3.71 and~7.23]{Watson}). Since
\[
 m_d'(t)
 =-\frac{2^{1-\nu}}{\Gamma(\nu)}
   t^\nu K_{\nu-1}(t),
\]
we obtain
\[
 |m_d'(t)|\leq C_d
 \begin{cases}
  t(1+|\log t|),&0<t\leq1,\ d=2,\\
  t,&0<t\leq1,\ d>2,\\
  t^{d/2-1/2}e^{-t},&t\geq1,
 \end{cases}
\]
where $C_d>0$ is another constant depending only on $d$.
Finally, since $m_d'(s)<0$ for $s>0$, we have
\[
\begin{aligned}
 1-m_d(t)
 &=m_d(0)-m_d(t)\\
 &=-\int_0^t m_d'(s)\dd s\\
 &=\int_0^t|m_d'(s)|\dd s.
\end{aligned}
\]
If $0<t\leq1$ and $d>2$, then
\[
\begin{aligned}
 1-m_d(t)
 &\leq C_d\int_0^t s\dd s\\
 &=\frac{C_d}{2}t^2.
\end{aligned}
\]
If $0<t\leq1$ and $d=2$, then
\[
\begin{aligned}
 1-m_2(t)
 &\leq C_2\int_0^t s(1+|\log s|)\dd s\\
 &=C_2t^2
   \left(\frac34+\frac12|\log t|\right)\\
 &\leq C_2t^2(1+|\log t|).
\end{aligned}
\]
For $t\geq1$, part~\ref{item:Bessel-positive} gives
\[
 0\leq1-m_d(t)\leq1.
\]
Hence there is a constant $C_d'>0$, depending only on $d$, such
that
\[
 0\leq1-m_d(t)\leq C_d'
 \begin{cases}
 t^2(1+|\log t|),&0<t\leq1,\ d=2,\\
 t^2,&0<t\leq1,\ d>2,\\
 1,&t\geq1.
 \end{cases}
\]
This proves part~\ref{item:Bessel-pointwise}.

It remains to prove \eqref{eq:dyadicBessel} and
\eqref{eq:dyadicApprox}, which constitute
part~\ref{item:Bessel-dyadic}. Suppose first that
$0<r<1$, and choose the integer $J\geq0$ such that
\[
    2^Jr\leq1<2^{J+1}r.
\]

For $d>2$, part~\ref{item:Bessel-pointwise} gives
\begin{align}
    \sum_{1\leq j\leq J}|m_d'(2^jr)|
    &\leq C_dr\sum_{1\leq j\leq J}2^j
    \notag\\
    &\leq 2C_d,
    \label{eq:dyadicBessel-small-dgt2}
\end{align}

where the last inequality follows from $2^Jr\leq1$.

For $d=2$ and $1\leq j\leq J$, write
\[
    j=J-k,\qquad 0\leq k\leq J-1.
\]
Since
\[
    \frac12<2^Jr\leq1,
\]
we have
\[
    2^jr=2^{-k}2^Jr\leq2^{-k}
\]
and
\[
\begin{aligned}
    |\log(2^jr)|
    &=|\log(2^Jr)-k\log2|\\
    &=|\log(2^Jr)|+k\log2\\
    &\leq(k+1)\log2.
\end{aligned}
\]
Hence
\[
    1+|\log(2^jr)|
    \leq(1+\log2)(1+k).
\]
Therefore, part~\ref{item:Bessel-pointwise} gives

\begin{align}
    \sum_{1\leq j\leq J}|m_2'(2^jr)|
    &\leq C_2\sum_{1\leq j\leq J}
       2^jr\bigl(1+|\log(2^jr)|\bigr)
       \notag\\
    &\leq(1+\log2)C_2
       \sum_{k=0}^{J-1}2^{-k}(1+k)
    \notag\\
    &\leq4(1+\log2)C_2,
    \label{eq:dyadicBessel-small-d2}
\end{align}

where we used
\[
    \sum_{k=0}^{\infty}2^{-k}(1+k)=4.
\]

For $j>J$, write
\[
    j=J+1+k,\qquad k\geq0,
\]
and put
\[
    a=2^{J+1}r.
\]
Then $1<a\leq2$ and
\[
    2^jr=2^ka.
\]
Set
\[
    \beta=\frac{d-1}{2}
\]
and
\[
    M_d=\sup_{x\geq1}x^\beta e^{-x/2}<\infty.
\]
Since $2^ka\geq1$, we have
\[
\begin{aligned}
    (2^ka)^\beta e^{-2^ka}
    &=
    \bigl((2^ka)^\beta e^{-2^{k-1}a}\bigr)
    e^{-2^{k-1}a}\\
    &\leq M_de^{-2^{k-1}a}\\
    &\leq M_de^{-2^{k-1}}.
\end{aligned}
\]
Consequently, part~\ref{item:Bessel-pointwise} gives

\begin{align}
    \sum_{j>J}|m_d'(2^jr)|
    &\leq C_d\sum_{k=0}^{\infty}
       (2^ka)^{d/2-1/2}e^{-2^ka}
        \notag\\
    &\leq B_d^{(1)}
       \sum_{k=0}^{\infty}e^{-2^{k-1}}
        \notag\\
    &\leq D_d^{(1)},
    \label{eq:dyadicBessel-large-tail}
\end{align}
where
\[
    B_d^{(1)}=C_dM_d
    \qquad\text{and}\qquad
    D_d^{(1)}
    =
    B_d^{(1)}
    \sum_{k=0}^{\infty}e^{-2^{k-1}}.
\]
Both constants depend only on $d$. The last series converges because
\[
    2^{k-1}\geq\frac{k}{2},
    \qquad
    e^{-2^{k-1}}\leq e^{-k/2}.
\]

If $r\geq1$, then $2^jr\geq1$ for every $j\geq1$. Using the same
constant $M_d$, we obtain
\[
    (2^jr)^\beta e^{-2^jr}
    \leq M_de^{-2^{j-1}r}.
\]
Therefore, part~\ref{item:Bessel-pointwise} gives

\begin{align}
    \sum_{j=1}^{\infty}|m_d'(2^jr)|
    &\leq C_d\sum_{j=1}^{\infty}
       (2^jr)^{(d-1)/2}e^{-2^jr}
       \notag\\
    &\leq B_d^{(2)}
       \sum_{j=1}^{\infty}e^{-2^{j-1}r}
       \notag\\
    &\leq B_d^{(2)}
       \sum_{j=1}^{\infty}e^{-2^{j-1}}
       \notag\\
    &\leq D_d^{(2)},
    \label{eq:dyadicBessel-r-large}
\end{align}

where
\[
    B_d^{(2)}=C_dM_d
    \qquad\text{and}\qquad
    D_d^{(2)}
    =
    B_d^{(2)}
    \sum_{j=1}^{\infty}e^{-2^{j-1}}.
\]
Again, both constants depend only on $d$.

For $0<r<1$ and $d=2$, combining
\eqref{eq:dyadicBessel-small-d2} and
\eqref{eq:dyadicBessel-large-tail} gives
\[
    \sum_{j=1}^{\infty}|m_2'(2^jr)|
    \leq
    4(1+\log2)C_2+D_2^{(1)}.
\]
For $0<r<1$ and $d>2$, combining
\eqref{eq:dyadicBessel-small-dgt2} and
\eqref{eq:dyadicBessel-large-tail} gives
\[
    \sum_{j=1}^{\infty}|m_d'(2^jr)|
    \leq
    2C_d+D_d^{(1)}.
\]
For $r\geq1$, estimate
\eqref{eq:dyadicBessel-r-large} gives
\[
    \sum_{j=1}^{\infty}|m_d'(2^jr)|
    \leq D_d^{(2)}.
\]
Thus one may take
\[
 C_d''=
 \begin{cases}
 \max\bigl\{
  4(1+\log2)C_2+D_2^{(1)},\,D_2^{(2)}
 \bigr\},&d=2,\\
 \max\bigl\{
  2C_d+D_d^{(1)},\,D_d^{(2)}
 \bigr\},&d>2.
 \end{cases}
\]
This proves \eqref{eq:dyadicBessel}.

We next prove \eqref{eq:dyadicApprox}. Suppose first that
$0<r<1$, and choose the integer $J\geq0$ such that
\[
    2^Jr\leq1<2^{J+1}r.
\]

If $d>2$, then part~\ref{item:Bessel-pointwise} gives

\begin{align}
    \sum_{1\leq j\leq J}
       2^{-j}|1-m_d(2^jr)|
    &\leq
       C_d'\sum_{1\leq j\leq J}2^{-j}(2^jr)^2
       \notag\\
    &=C_d'r^2\sum_{1\leq j\leq J}2^j
    \notag\\
    &\leq2C_d'r,
    \label{eq:dyadicApprox-small-dgt2}
\end{align}

where the last inequality follows from
\[
    \sum_{j=1}^{J}2^j<2^{J+1}
    \qquad\text{and}\qquad
    2^Jr\leq1.
\]

If $d=2$, write
\[
    j=J-k,\qquad 0\leq k\leq J-1.
\]
As above,
\[
    2^jr\leq2^{-k}
\]
and
\[
    1+|\log(2^jr)|
    \leq(1+\log2)(1+k).
\]
Moreover,
\[
    2^{-j}(2^jr)^2
    =2^jr^2
    =r(2^Jr)2^{-k}
    \leq r2^{-k}.
\]
Hence part~\ref{item:Bessel-pointwise} gives

\begin{align}
    \sum_{1\leq j\leq J}
       2^{-j}|1-m_2(2^jr)|
    &\leq
       C_2'\sum_{1\leq j\leq J}
       2^jr^2\bigl(1+|\log(2^jr)|\bigr)
        \notag\\
    &\leq
       (1+\log2)C_2'r
       \sum_{k=0}^{J-1}2^{-k}(1+k)
        \notag\\
    &\leq
       4(1+\log2)C_2'r.
       \label{eq:dyadicApprox-small-d2}
\end{align}

For $j>J$, we have \(2^jr>1\). Part~\ref{item:Bessel-positive}
therefore gives
\[
    0\leq1-m_d(2^jr)\leq1.
\]
Consequently,

\begin{align}
    \sum_{j>J}2^{-j}|1-m_d(2^jr)|
    &\leq\sum_{j>J}2^{-j}
    \notag\\
    &=2^{-J}
    \notag\\
    &<2r,
    \label{eq:dyadicApprox-large-tail}
\end{align}
where the last inequality follows from
\[
    1<2^{J+1}r.
\]

For $0<r<1$ and $d=2$, combining
\eqref{eq:dyadicApprox-small-d2} and
\eqref{eq:dyadicApprox-large-tail} gives
\[
\begin{aligned}
    \sum_{j=1}^{\infty}
       2^{-j}|1-m_2(2^jr)|
    &\leq
       \bigl(4(1+\log2)C_2'+2\bigr)r.
\end{aligned}
\]
For $0<r<1$ and $d>2$, combining
\eqref{eq:dyadicApprox-small-dgt2} and
\eqref{eq:dyadicApprox-large-tail} gives
\[
\begin{aligned}
    \sum_{j=1}^{\infty}
       2^{-j}|1-m_d(2^jr)|
    &\leq
       (2C_d'+2)r.
\end{aligned}
\]

If $r\geq1$, then part~\ref{item:Bessel-positive} gives
\[
\begin{aligned}
    \sum_{j=1}^{\infty}
       2^{-j}|1-m_d(2^jr)|
    &\leq\sum_{j=1}^{\infty}2^{-j}\\
    &=1\\
    &\leq r.
\end{aligned}
\]

Thus one may take
\[
    C_d'''=
    \begin{cases}
      \max\bigl\{4(1+\log2)C_2'+2,\,1\bigr\},
      &d=2,\\
      \max\bigl\{2C_d'+2,\,1\bigr\},
      &d>2.
    \end{cases}
\]
Therefore,
\[
    \sum_{j=1}^{\infty}
       2^{-j}|1-m_d(2^jr)|
    \leq C_d'''r,
    \qquad r>0.
\]
This proves \eqref{eq:dyadicApprox} and completes the proof of
part~\ref{item:Bessel-dyadic}.
\end{proof}
The estimates \eqref{eq:dyadicBessel} and
\eqref{eq:dyadicApprox} will be used in Section~3 in the proof of
Proposition~\ref{prop:counterexample}.
\subsection{Auxiliary H\"older estimates}

For $f\in L^\infty(\R^d;\R^m)$, define its hyperbolic Poisson
extension by the Poisson integral
(see \cite[Section~5.6, p.~78]{Stoll2016}):
\[
  \Pcal[f](x,y)
  =
  (\Pcal_y*f)(x)
  =
  \int_{\R^d}\Pcal_y(x-z)f(z)\dd z,
\]
where $*$ denotes the usual convolution on $\R^d$.
A direct computation gives
\[
  \Delta_h\Pcal[f]=0
  \qquad\text{in }\Hh^n,
\]
and hence $\Pcal[f]$ is hyperbolic harmonic. Here \(\Delta_h\) denotes the hyperbolic Laplacian defined in the Introduction.

If a bounded hyperbolic harmonic mapping
$u:\Hh^n\to\R^m$ has a uniform boundary trace $f$,
then
\[
  u(x,y)=\Pcal[f](x,y).
\]
Indeed, $\Pcal[f]$ is bounded and hyperbolic harmonic
with boundary values $f$
(see \cite[Theorem~5.6.2]{Stoll2016}).
Thus $u-\Pcal[f]$ is bounded and has zero boundary values
on $\R^d\times\{0\}$, the finite boundary of $\Hh^n$.
The half-space-to-ball map, the invariant mean-value
property, and bounded convergence then give
$u-\Pcal[f]=0$ on $\Hh^n$
(see \cite[(2.3.2) and Corollary~4.1.3]{Stoll2016}).

We next introduce a homogeneous Littlewood--Paley decomposition.
Choose a radial function
$\psi\in C_c^\infty(\R^d\setminus\{0\})$ such that
\begin{equation}\label{eq:psi-support}
 \supp\psi
 \subset
 \{\xi\in\R^d:1/2\le|\xi|\le2\}
\end{equation}
and
\[
 \sum_{j\in\mathbb Z}\psi(2^{-j}\xi)=1,
 \qquad \xi\in\R^d\setminus\{0\}.
\]
Such a choice follows from
\cite[p.~455]{Grafakos}, where the function
appearing in this identity is the Fourier transform of a radial
Schwartz function. For $g\in L^\infty(\R^d;\R^m)$, define the
Littlewood--Paley operator
$\Delta_jg$ by
\begin{equation}\label{eq:dyadic-definition}
 \widehat{\Delta_jg}(\xi)
 =
 \psi(2^{-j}\xi)\widehat g(\xi),
 \qquad j\in\mathbb Z.
\end{equation}
Equivalently,
\begin{equation}\label{eq:dyadic-convolution}
  \Delta_jg=k_j*g,
\end{equation}
where
\begin{equation}\label{eq:dyadic-kernel}
  k_j(x)=2^{jd}\check\psi(2^jx),
  \qquad x\in\R^d,\quad j\in\mathbb Z
\end{equation}
(see \cite[Definition~6.1.1]{Grafakos}). For vector-valued $g$,
we apply $\Delta_j$ to each component.

For $0<\alpha\le1$ and a mapping $g:\R^d\to\R^m$, set
\[
 [g]_{\alpha,\R^d}
 =
 \sup_{\substack{x,z\in\R^d\\x\ne z}}
 \frac{|g(x)-g(z)|}{|x-z|^\alpha}.
\]
We denote by $\Lambda_\alpha(\R^d;\R^m)$ the class of all
continuous mappings $g:\R^d\to\R^m$ such that
$[g]_{\alpha,\R^d}<\infty$. When $\alpha=1$, we also write
\[
 \Lip(\R^d;\R^m)=\Lambda_1(\R^d;\R^m)
\]
and
\[
 [g]_{\Lip(\R^d)}=[g]_{1,\R^d}.
\]
The following vector-valued form of Bernstein's inequality
follows from the proof of
\cite[Lemma~2.1]{BahouriCheminDanchin}.

\begin{lemma}[Bernstein's inequality]\label{lem:Bernstein}
Let $R>0$, and let $h\in L^\infty(\R^d;\R^m)$ satisfy
\[
  \operatorname{supp}\widehat h
  \subset\{\xi\in\R^d:|\xi|\le R\}.
\]
Then, for every multi-index $\beta$,
\[
  \|\partial^\beta h\|_{L^\infty}
  \le C_{d,\beta}R^{|\beta|}\|h\|_{L^\infty},
\]
where $C_{d,\beta}$ is independent of $m$.
\end{lemma}

\begin{proof}
Choose $\varphi$ such that
$\widehat\varphi\in C_c^\infty(\R^d)$ and
$\widehat\varphi=1$ on $\{|\xi|\le1\}$. Set
$\varphi_{1/R}(x)=R^d\varphi(Rx)$.
Since $\widehat{\varphi_{1/R}}=1$ on
$\operatorname{supp}\widehat h$, we have
$h=h*\varphi_{1/R}$ and hence
$\partial^\beta h=h*\partial^\beta\varphi_{1/R}$.
Thus
\[
  \|\partial^\beta h\|_{L^\infty}
  \le \|h\|_{L^\infty}
      \|\partial^\beta\varphi_{1/R}\|_{L^1}
  = R^{|\beta|}\|\partial^\beta\varphi\|_{L^1}
      \|h\|_{L^\infty}.
\]
The estimate uses the Euclidean norm of $h$, so its
constant is independent of $m$.
\end{proof}
With this preparation, we establish the inverse estimate.

\begin{lemma}\label{lem:inverse}
Let $d\ge2$, $0<\alpha<1$, and
$g\in L^\infty(\R^d;\R^m)$. Then
\begin{equation}\label{eq:approxNorm}
  A_\alpha(g)
  =
  \sup_{y>0}y^{-\alpha}
  \norm{\Pcal_y*g-g}_{L^\infty}
  <\infty
\end{equation}
if and only if $g$ has a representative in
$\Lambda_\alpha(\R^d;\R^m)$. Quantitatively,
\[
  A_{d,\alpha}^{-1}[g]_{\alpha,\R^d}
  \le A_\alpha(g)
  \le A_{d,\alpha}[g]_{\alpha,\R^d},
\]
where $A_{d,\alpha}>0$ depends only on $d$ and $\alpha$.
\end{lemma}

\begin{proof}
Assume first that
$g\in\Lambda_\alpha(\R^d;\R^m)$. By a change of variables and \eqref{eq:poisson-normalization},
\[
  (\Pcal_y*g)(x)-g(x)
  =
  \int_{\R^d}\Pcal_y(z)
  \bigl(g(x-z)-g(x)\bigr)\dd z.
\]
Hence
\[
\begin{aligned}
  |(\Pcal_y*g)(x)-g(x)|
  &\le
  [g]_{\alpha,\R^d}
  \int_{\R^d}\Pcal_y(z)|z|^\alpha\dd z\\
  &=
  y^\alpha[g]_{\alpha,\R^d}
  \int_{\R^d}\Pcal_1(w)|w|^\alpha\dd w,
\end{aligned}
\]
where the equality follows from \eqref{eq:poisson-dilation}
and the change of variables $z=yw$.
The last integral is finite because $\Pcal_1$ is bounded near the
origin and
$\Pcal_1(w)=O(|w|^{-2d})$ as $|w|\to\infty$, while $\alpha<d$.
Set
\[
  C_{d,\alpha}^{(1)}
  =
  \int_{\R^d}\Pcal_1(w)|w|^\alpha\dd w.
\]
Taking the supremum first over $x\in\R^d$ and then over $y>0$, we
obtain
\begin{equation}\label{eq:directApprox}
  A_\alpha(g)
  \le
  C_{d,\alpha}^{(1)}[g]_{\alpha,\R^d}.
\end{equation}

Conversely, suppose that \eqref{eq:approxNorm} holds.
Recall that
$\supp\psi\subset\{\xi\in\R^d:1/2\le|\xi|\le2\}$.
By Lemma~\ref{lem:Bessel}\textup{(i)},
$1-m_d(|\xi|)>0$ on $\supp\psi$.
Since $m_d$ is smooth on $(0,\infty)$, it follows that
\[
  a(\xi)=\frac{\psi(\xi)}{1-m_d(|\xi|)},
  \qquad \xi\in\R^d\setminus\{0\},
  \qquad a(0)=0,
\]
belongs to $C_c^\infty(\R^d)$, with
$\supp a\subset\supp\psi$.

Let $\check a$ denote the inverse Fourier transform of $a$,
and set
\[
  K_j(x)=2^{jd}\check a(2^jx),
  \qquad x\in\R^d,\quad j\in\mathbb Z.
\]
Since $a$ is smooth and compactly supported,
$\check a\in L^1(\R^d)$
(see \cite[Proposition~2.2.6 and Corollary~2.2.15]{Grafakos}).
Thus $K_j\in L^1(\R^d)$.

Recall that $k_j$ is defined by \eqref{eq:dyadic-kernel}.
Similarly, $k_j\in L^1(\R^d)$.
The Fourier scaling and convolution formulas, together
with \eqref{eq:multiplier}, give
\[
\begin{aligned}
  \widehat{K_j}(\xi)
  -\widehat{K_j}(\xi)\widehat{\Pcal_{2^{-j}}}(\xi)
  &=a(2^{-j}\xi)
    \bigl(1-m_d(2^{-j}|\xi|)\bigr)\\
  &=\psi(2^{-j}\xi)
   =\widehat{k_j}(\xi).
\end{aligned}
\]
By the injectivity of the Fourier transform on $L^1(\R^d)$,
\begin{equation}\label{eq:inverse-kernel-identity}
  K_j-K_j*\Pcal_{2^{-j}}=k_j.
\end{equation}

For $j\in\mathbb Z$, put
\begin{equation}\label{eq:dyadic-error}
  h_j=g-\Pcal_{2^{-j}}*g.
\end{equation}
Hence, by \eqref{eq:dyadic-convolution},
\eqref{eq:inverse-kernel-identity},
\eqref{eq:dyadic-error}, and Fubini's theorem,
\[
\begin{aligned}
  \Delta_jg
  &=k_j*g\\
  &=\bigl(K_j-K_j*\Pcal_{2^{-j}}\bigr)*g\\
  &=K_j*h_j
   =K_j*\bigl(g-\Pcal_{2^{-j}}*g\bigr).
\end{aligned}
\]
By the change of variables $w=2^jx$,
\[
\begin{aligned}
  \norm{K_j}_{L^1}
  &=\int_{\R^d}2^{jd}|\check a(2^jx)|\dd x\\
  &=\int_{\R^d}|\check a(w)|\dd w\\
  &=\norm{\check a}_{L^1}.
\end{aligned}
\]
Thus, Young's inequality gives
\[
  \norm{\Delta_jg}_{L^\infty}
  \le \norm{K_j}_{L^1}\norm{h_j}_{L^\infty}
  =\norm{\check a}_{L^1}\norm{h_j}_{L^\infty},
  \qquad j\in\mathbb Z.
\]

We next estimate $h_j$. By the definition of $A_\alpha(g)$,
\[
  \norm{\Pcal_y*g-g}_{L^\infty}
  \le
  A_\alpha(g)y^\alpha,
  \qquad y>0.
\]
Taking $y=2^{-j}$ and using \eqref{eq:dyadic-error},
we obtain
\[
  \norm{h_j}_{L^\infty}
  \le
  A_\alpha(g)2^{-j\alpha}.
\]
Since $\Delta_jg=K_j*h_j$, the preceding estimate and Young's
inequality give
\begin{equation}\label{eq:blockHolder}
  \norm{\Delta_jg}_{L^\infty}
  \le B_d^{(1)}A_\alpha(g)2^{-j\alpha},
  \qquad j\in\mathbb Z.
\end{equation}
Here $B_d^{(1)}:=\norm{\check a}_{L^1}$ depends only on $d$
because $a$ is determined by $m_d$ and the function $\psi$
fixed above.

Let $v\in\R^d\setminus\{0\}$ and choose $J\in\mathbb Z$ such that
\[
  2^J|v|\le1<2^{J+1}|v|.
\]
We first estimate the terms with $j\le J$. By the definition of
$\Delta_j$ and the support condition on $\psi$,
\[
  \supp\widehat{\Delta_jg}
  \subset
  \{\xi\in\R^d:2^{j-1}\le|\xi|\le2^{j+1}\}.
\]
It follows from Lemma~\ref{lem:Bernstein}, applied with
$R=2^{j+1}$, that
\[
  \norm{\nabla\Delta_jg}_{L^\infty}
  \le
  B_d^{(2)}2^j\norm{\Delta_jg}_{L^\infty},
\]
where $B_d^{(2)}>0$ depends only on $d$.
The mean value theorem and \eqref{eq:blockHolder} therefore give
\[
\begin{aligned}
  \sup_{x\in\R^d}
  |\Delta_jg(x+v)-\Delta_jg(x)|
  &\le
  |v|\norm{\nabla\Delta_jg}_{L^\infty}\\
  &\le
  B_d^{(2)}2^j|v|\norm{\Delta_jg}_{L^\infty}\\
  &\le
  B_d^{(1)}B_d^{(2)}
  A_\alpha(g)2^{j(1-\alpha)}|v|.
\end{aligned}
\]
Summing over $j\le J$, we obtain
\begin{align}
  \sum_{j\le J}\sup_{x\in\R^d}
  |\Delta_jg(x+v)-\Delta_jg(x)|
  &\le
  B_d^{(1)}B_d^{(2)}A_\alpha(g)|v|
  \sum_{j\le J}2^{j(1-\alpha)}
  \notag\\
  &=
  \frac{B_d^{(1)}B_d^{(2)}}{1-2^{-(1-\alpha)}}
  A_\alpha(g)|v|2^{J(1-\alpha)}
  \notag\\
  &\le
  \frac{B_d^{(1)}B_d^{(2)}}{1-2^{-(1-\alpha)}}
  A_\alpha(g)|v|^\alpha.
  \label{eq:lowFrequency}
\end{align}
Here the last inequality follows from $2^J|v|\le1$, since
\[
  |v|2^{J(1-\alpha)}
  =
  |v|^\alpha(2^J|v|)^{1-\alpha}
  \le
  |v|^\alpha.
\]

For the terms with $j>J$, the triangle inequality and
\eqref{eq:blockHolder} give
\begin{align*}
  \sum_{j>J}\sup_{x\in\R^d}
  |\Delta_jg(x+v)-\Delta_jg(x)|
  &\le
  2\sum_{j>J}\norm{\Delta_jg}_{L^\infty}\\
  &\le
  2B_d^{(1)}A_\alpha(g)\sum_{j>J}2^{-j\alpha}\\
  &=
  \frac{2B_d^{(1)}}{2^\alpha-1}
  A_\alpha(g)2^{-J\alpha}.
\end{align*}
Since $1<2^{J+1}|v|$, we have
$2^{-J}<2|v|$ and hence
$2^{-J\alpha}\le2^\alpha|v|^\alpha$. Therefore,
\begin{equation}\label{eq:highFrequency}
  \sum_{j>J}\sup_{x\in\R^d}
  |\Delta_jg(x+v)-\Delta_jg(x)|
  \le
  \frac{2^{1+\alpha}B_d^{(1)}}{2^\alpha-1}
  A_\alpha(g)|v|^\alpha.
\end{equation}

Set
\[
  C_{d,\alpha}^{(2)}
  =
  \frac{B_d^{(1)}B_d^{(2)}}{1-2^{-(1-\alpha)}}
  +
  \frac{2^{1+\alpha}B_d^{(1)}}{2^\alpha-1}.
\]
Combining \eqref{eq:lowFrequency} and \eqref{eq:highFrequency},
we obtain, for every $v\in\R^d$,
\[
  \sum_{j\in\mathbb Z}
  \sup_{x\in\R^d}
  \bigl|\Delta_jg(x+v)-\Delta_jg(x)\bigr|
  \le
  C_{d,\alpha}^{(2)}A_\alpha(g)|v|^\alpha.
\]
Define
\[
G(x)
=
\sum_{j\in\mathbb Z}
\bigl(\Delta_jg(x)-\Delta_jg(0)\bigr),
\qquad x\in\R^d.
\]
The preceding estimate shows that the series converges locally
uniformly. It also gives
\[
|G(x+v)-G(x)|
\le
C_{d,\alpha}^{(2)}A_\alpha(g)|v|^\alpha,
\qquad x,v\in\R^d.
\]
Since $G(0)=0$, we further have
\[
|G(x)|
\le
C_{d,\alpha}^{(2)}A_\alpha(g)|x|^\alpha,
\qquad x\in\R^d.
\]

By the homogeneous Littlewood--Paley decomposition
(see \cite[p.~16]{Bahouri2019}), there exists an
$\R^m$-valued polynomial $Q$ on $\R^d$ such that
\[
g(x)=G(x)+Q(x)
\]
for almost every $x\in\R^d$. Since $g$ is bounded, the preceding estimate yields
\[
  |Q(x)|
  \le
  \norm{g}_{L^\infty}
  +
  C_{d,\alpha}^{(2)}A_\alpha(g)|x|^\alpha
\]
for almost every $x\in\R^d$. By continuity, this inequality
holds for every $x\in\R^d$.
Since $0<\alpha<1$, this growth estimate implies that
the polynomial $Q$ is constant.

Therefore, $G+Q$ is a representative of $g$.
Identifying $g$ with this representative, we obtain
\[
|g(x+v)-g(x)|
\le
C_{d,\alpha}^{(2)}A_\alpha(g)|v|^\alpha,
\qquad x,v\in\R^d.
\]
Hence
\[
[g]_{\alpha,\R^d}
\le
C_{d,\alpha}^{(2)}A_\alpha(g).
\]

Finally, let
\[
A_{d,\alpha}
=
\max\bigl\{
C_{d,\alpha}^{(1)},C_{d,\alpha}^{(2)}
\bigr\}.
\]
Then \eqref{eq:directApprox} and the preceding estimate give
\[
A_{d,\alpha}^{-1}[g]_{\alpha,\R^d}
\le
A_\alpha(g)
\le
A_{d,\alpha}[g]_{\alpha,\R^d}.
\]
This completes the proof.
\end{proof}
We next record the corresponding direct regularity estimate for the
hyperbolic Poisson extension.
\begin{lemma}\label{lem:extension}
Let $d\ge2$ and $0<\alpha\le1$. If
$g\in L^\infty(\R^d;\R^m)\cap\Lambda_\alpha(\R^d;\R^m)$, then
\[
 G(x,y)=(\Pcal_y*g)(x)
\]
extends continuously to $\overline{\Hh^n}$ and satisfies
\[
 [G]_{\alpha,\Hh^n}
 \le C_{d,\alpha}[g]_{\alpha,\R^d}.
\]
In particular, when $\alpha=1$,
\[
 [G]_{\Lip(\overline{\Hh^n})}
 \le C_d[g]_{\Lip(\R^d)}.
\]
\end{lemma}

\begin{proof}
By the change of variables $z=yw$ with $z,w\in\R^d$
and the dilation identity
$\Pcal_y(z)=y^{-d}\Pcal_1(z/y)$, we have
\[
 G(x,y)
 =
 \int_{\R^d}\Pcal_1(w)g(x-yw)\dd w,
 \qquad x\in\R^d,\quad y>0.
\]
Since $\Pcal_1(w)=c_d(1+|w|^2)^{-d}$ and
$0<\alpha\le1<d$, polar coordinates give
\[
\begin{aligned}
 M_{d,\alpha}
 &:=
 \int_{\R^d}\Pcal_1(w)|w|^\alpha\dd w\\
 &\le
 C_d\left(
 \int_0^1 r^{d+\alpha-1}\dd r
 +
 \int_1^\infty r^{\alpha-d-1}\dd r
 \right)
 <\infty.
\end{aligned}
\]

We first estimate the horizontal increments. For
$x,x'\in\R^d$ and $y>0$, the positivity and normalization
of $\Pcal_1$ yield
\begin{equation}\label{eq:extensionHorizontal}
\begin{aligned}
 |G(x,y)-G(x',y)|
 &\le
 \int_{\R^d}\Pcal_1(w)
 |g(x-yw)-g(x'-yw)|\dd w\\
 &\le
 [g]_{\alpha,\R^d}|x-x'|^\alpha.
\end{aligned}
\end{equation}

Next, since $\Pcal_1>0$ and
$\int_{\R^d}\Pcal_1(w)\dd w=1$, we have
\[
\begin{aligned}
 |G(x,y)-g(x)|
 &\le
 \int_{\R^d}\Pcal_1(w)
 |g(x-yw)-g(x)|\dd w\\
 &\le
 M_{d,\alpha}[g]_{\alpha,\R^d}y^\alpha,
\end{aligned}
\]
where the last inequality follows from the assumption
$g\in\Lambda_\alpha(\R^d;\R^m)$
and the definition of $M_{d,\alpha}$.

Define $G(x,0)=g(x)$. For each $x_0\in\R^d$, we obtain
\[
\begin{aligned}
 |G(x,y)-g(x_0)|
 &\le
 |G(x,y)-g(x)|+|g(x)-g(x_0)|\\
 &\le
 [g]_{\alpha,\R^d}
 \bigl(M_{d,\alpha}y^\alpha+|x-x_0|^\alpha\bigr)
 \longrightarrow0
\end{aligned}
\]
as $(x,y)\to(x_0,0)$ with $y>0$. This proves continuity
at the boundary. Moreover, the integral representation of $G$
and \eqref{eq:extensionHorizontal} also hold when $y=0$.

We now estimate the vertical increments. For $x\in\R^d$
and $s,t\ge0$, we have
\begin{align}
 |G(x,t)-G(x,s)|
 &\le
 \int_{\R^d}\Pcal_1(w)
 |g(x-tw)-g(x-sw)|\dd w
\notag\\
 &\le
 [g]_{\alpha,\R^d}|t-s|^\alpha
 \int_{\R^d}\Pcal_1(w)|w|^\alpha\dd w
 \notag\\
 &=
 M_{d,\alpha}[g]_{\alpha,\R^d}|t-s|^\alpha.
 \label{eq:extensionVertical}
\end{align}

Finally, let $X=(x,s)$ and $Y=(x',t)$ belong to
$\overline{\Hh^n}$. Applying
\eqref{eq:extensionHorizontal} and
\eqref{eq:extensionVertical}, we obtain
\[
\begin{aligned}
 |G(X)-G(Y)|
 &\le
 |G(x,s)-G(x',s)|
 +
 |G(x',s)-G(x',t)|\\
 &\le
 [g]_{\alpha,\R^d}
 \bigl(|x-x'|^\alpha
       +M_{d,\alpha}|s-t|^\alpha\bigr)\\
 &\le
 (1+M_{d,\alpha})[g]_{\alpha,\R^d}|X-Y|^\alpha,
\end{aligned}
\]
where the last inequality follows from
$|x-x'|\le|X-Y|$ and $|s-t|\le|X-Y|$.
Thus $G$ is continuous on $\overline{\Hh^n}$.
Taking the supremum over distinct
$X,Y\in\overline{\Hh^n}$ gives
\[
 [G]_{\alpha,\Hh^n}
 \le
 C_{d,\alpha}[g]_{\alpha,\R^d},
\]
where $C_{d,\alpha}=1+M_{d,\alpha}$ depends only on
$d$ and $\alpha$. The case $\alpha=1$ gives the asserted
Lipschitz estimate.
\end{proof}

Lemma~\ref{lem:inverse} gives the inverse estimate for
$0<\alpha<1$, while Lemma~\ref{lem:extension} gives the
direct estimate for $0<\alpha\le1$. We next establish
the additional estimates needed at the endpoint $\alpha=1$.

\subsection{Endpoint estimates}

With $\psi$ chosen in Section~2.2, set
\[
 \varphi(\xi)
 =
 1-\sum_{j\ge0}\psi(2^{-j}\xi),
 \qquad \xi\in\R^d.
\]
The support and partition properties of $\psi$ imply that
$\varphi\in C_c^\infty(\R^d)$, where $C_c^\infty(\R^d)$
denotes the class of smooth functions with compact support
in $\R^d$.

For $g\in L^\infty(\R^d;\R^m)$, define $S_0g$ by
\begin{equation}\label{eq:S0-definition}
  \widehat{S_0g}(\xi)
  =\varphi(\xi)\widehat g(\xi),
  \qquad \xi\in\R^d.
\end{equation}
The corresponding Littlewood--Paley decomposition is
\[
 g=S_0g+\sum_{j\ge0}\Delta_jg,
\]
where the Fourier identities and the convergence are understood
in the sense of tempered distributions
(see \cite[Exercises~2.3.12(a) and~6.1.10]{Grafakos}).

For two nonnegative quantities $A$ and $B$, we write
$A\asymp_d B$ if $c_dA\le B\le C_dA$ for some positive
constants $c_d$ and $C_d$ depending only on $d$. Recall from Section~1 that
\[
 \omega_2(g,t)_\infty
 =
 \sup_{x\in\R^d}\sup_{|h|\le t}
 |g(x+h)-2g(x)+g(x-h)|,
 \qquad t>0,
\]
and
\[
 \Dcal(g)
 =
 \int_0^1
 \frac{\omega_2(g,t)_\infty}{t^2}\dd t.
\]
The inhomogeneous Besov space
$B_{\infty,1}^{1}(\R^d;\R^m)$ consists of all bounded
continuous mappings $g:\R^d\to\R^m$ satisfying
$\Dcal(g)<\infty$.
Its dyadic characterization gives
\begin{equation}\label{eq:BesovDyadic}
 \norm{g}_{L^\infty}+\Dcal(g)
 \asymp_d
 \norm{g}_{L^\infty}
 +\sum_{j\ge0}2^j\norm{\Delta_jg}_{L^\infty}
\end{equation}
for every bounded continuous mapping
$g:\R^d\to\R^m$
(see \cite[Section~2.3.1 and Theorem~2.5.12]{Triebel1983}).
The next lemma gives the consequence of
\eqref{eq:BesovDyadic} needed below.

\begin{lemma}
\label{lem:DZLip}
Suppose that $g\in B_{\infty,1}^{1}(\R^d;\R^m)$.
Then $g$ is Lipschitz continuous and
\[
 [g]_{\Lip(\R^d)}
 \le
 C_d\bigl(
 \norm{g}_{L^\infty}+\Dcal(g)
 \bigr),
\]
where $C_d>0$ depends only on $d$.
\end{lemma}

\begin{proof}
It follows from \eqref{eq:psi-support} and
\eqref{eq:dyadic-definition} that
\[
\begin{aligned}
 \supp\widehat{\Delta_jg}
 &\subset \{\xi\in\R^d:2^{-j}\xi\in\supp\psi\}\\
 &\subset \{\xi\in\R^d:2^{j-1}\le|\xi|\le2^{j+1}\}\\
 &\subset \{\xi\in\R^d:|\xi|\le2^{j+1}\}.
\end{aligned}
\]
Applying Lemma~\ref{lem:Bernstein} with
$h=\Delta_jg$, $R=2^{j+1}$ and $|\beta|=1$, we obtain
\[
 \norm{\nabla\Delta_jg}_{L^\infty}
 \le
 C_{d,1}2^j\norm{\Delta_jg}_{L^\infty},
\]
where $C_{d,1}>0$ depends only on $d$.
By \eqref{eq:BesovDyadic}, we have
\begin{equation}\label{eq:weighted-dyadic-bound}
 \sum_{j\ge0}2^j\norm{\Delta_jg}_{L^\infty}
 \le C_{d,2}\bigl(\norm{g}_{L^\infty}+\Dcal(g)\bigr)<\infty,
\end{equation}
where $C_{d,2}>0$ depends only on $d$
for the fixed choice of $\psi$.
Summing over $j\ge0$ and using
\eqref{eq:weighted-dyadic-bound}, we obtain
\begin{equation}\label{eq:gradientDelta}
 \sum_{j\ge0}\norm{\nabla\Delta_jg}_{L^\infty}
 \le
 C_{d,1}C_{d,2}\bigl(
 \norm{g}_{L^\infty}+\Dcal(g)
 \bigr).
\end{equation}

Next, we estimate $\nabla S_0g$.
Let $L$ be defined by $\widehat L=\varphi$.
Since $\varphi\in C_c^\infty(\R^d)$, we have
$\nabla L\in L^1(\R^d)$
(see \cite[Example~2.2.2, Corollary~2.2.15,
and Proposition~2.2.6]{Grafakos}).
Since $\widehat L=\varphi$, \eqref{eq:S0-definition}
and the Fourier convolution identity give $S_0g=L*g$.
Thus Young's inequality yields
\begin{align}
 \norm{\nabla S_0g}_{L^\infty}
 &=\norm{(\nabla L)*g}_{L^\infty}
 \notag\\
 &\le \norm{\nabla L}_{L^1}\norm{g}_{L^\infty}
 \notag\\
 &\le C_{d,3}\norm{g}_{L^\infty},
 \label{eq:gradientS0}
\end{align}
where $C_{d,3}>0$ depends only on $d$
for the fixed choice of $\psi$.

It follows from \eqref{eq:weighted-dyadic-bound} and
\eqref{eq:gradientDelta} that the series
$\sum_{j\ge0}\Delta_jg$ and
$\sum_{j\ge0}\nabla\Delta_jg$ converge uniformly on $\R^d$.
Thus, differentiating $g=S_0g+\sum_{j\ge0}\Delta_jg$
term by term, we obtain
$$
 \nabla g
 =\nabla S_0g+\sum_{j\ge0}\nabla\Delta_jg.$$
Combining \eqref{eq:gradientDelta} and
\eqref{eq:gradientS0}, we obtain
\[
\begin{aligned}
 \norm{\nabla g}_{L^\infty}
 &\le
 \norm{\nabla S_0g}_{L^\infty}
 +
 \sum_{j\ge0}\norm{\nabla\Delta_jg}_{L^\infty}\\
 &\le
 C_d\bigl(
 \norm{g}_{L^\infty}+\Dcal(g)
 \bigr),
\end{aligned}
\]
where $C_d=C_{d,1}C_{d,2}+C_{d,3}$.

For $x,z\in\R^d$, the fundamental theorem of calculus applied along
the line segment joining $z$ and $x$ gives
\[
 g(x)-g(z)
 =
 \int_0^1
 \nabla g\bigl(z+t(x-z)\bigr)(x-z)\dd t.
\]
Hence
\[
 |g(x)-g(z)|
 \le
 \norm{\nabla g}_{L^\infty}|x-z|.
\]
Dividing by $|x-z|$ and taking the supremum over all
$x,z\in\R^d$ with $x\ne z$, we obtain
\[
 [g]_{\Lip(\R^d)}
 \le
 \norm{\nabla g}_{L^\infty}
 \le
 C_d\bigl(
 \norm{g}_{L^\infty}+\Dcal(g)
 \bigr).
\]\end{proof}

The next lemma relates the boundary condition $\Dcal(g)<\infty$ to
the hyperbolic Poisson approximation used in
Theorem~\ref{thm:endpoint}.

\begin{lemma}\label{lem:approxBesov}
Let $d\ge2$ and let $g\in L^\infty(\R^d;\R^m)$ be continuous.  Then
\[
 \norm{g}_{L^\infty}+\Dcal(g)
 \asymp_d
 \norm{g}_{L^\infty}
 +\int_0^1
   \frac{\norm{\Pcal_y*g-g}_{L^\infty}}{y^2}\dd y.
\]
\end{lemma}

\begin{proof}
Set $E(y)=\norm{\Pcal_y*g-g}_{L^\infty}$.
By \eqref{eq:BesovDyadic}, it suffices to prove that
\[
 \norm{g}_{L^\infty}
 +\sum_{j\ge0}2^j\norm{\Delta_jg}_{L^\infty}
 \asymp_d
 \norm{g}_{L^\infty}
 +\int_0^1\frac{E(y)}{y^2}\dd y.
\]

We first prove that
\[
 \sum_{j\ge0}2^j\norm{\Delta_jg}_{L^\infty}
 \le
 C_{d,1}\int_0^1\frac{E(y)}{y^2}\dd y,
\]
where $C_{d,1}>0$ will be chosen below.
Fix $j\ge0$ and let
\[
 2^{-j-1}\le y\le2^{-j}.
\]
Put $t=2^jy$. Then $1/2\le t\le1$. Set
\[
 a_t(\xi)
 =
 \frac{\psi(\xi)}{1-m_d(t|\xi|)},
 \qquad \xi\in\R^d\setminus\{0\},
 \qquad a_t(0)=0.
\]
On the support of $\psi$, we have
\[
 \frac14\le t|\xi|\le2.
\]
Hence Lemma~\ref{lem:Bessel}\textup{(i)} shows that
$1-m_d(t|\xi|)$ is uniformly positive. The functions $a_t$ are supported in the fixed set
$\supp\psi$. By the smoothness of $m_d$ and the uniform
lower bound for the denominator, all their derivatives with respect to $\xi$
are bounded uniformly for $t\in[1/2,1]$.
Let $\check a_t$ denote the inverse Fourier transform of $a_t$.
The common compact support and the uniform derivative bounds give
\begin{equation}\label{eq:at-kernel-bound}
 C_{d,1}:=\sup_{1/2\le t\le1}
 \|\check a_t\|_{L^1}<\infty,
\end{equation}
where $C_{d,1}$ depends only on $d$
for the fixed choice of $\psi$.

Define
\[
 K_{j,t}(x)=2^{jd}\check a_t(2^jx),
 \qquad x\in\R^d.
\]
Then $K_{j,t}\in L^1(\R^d)$ and
\[
 \|K_{j,t}\|_{L^1}=\|\check a_t\|_{L^1}.
\]
Since $t=2^jy$, the Fourier scaling formula,
\eqref{eq:multiplier}, and \eqref{eq:dyadic-kernel} give
\[
\begin{aligned}
 \widehat{K_{j,t}}(\xi)
 -\widehat{K_{j,t}}(\xi)\widehat{\Pcal_y}(\xi)
 &=
 a_t(2^{-j}\xi)\bigl(1-m_d(y|\xi|)\bigr)\\
 &=\psi(2^{-j}\xi)\\
 &=\widehat{k_j}(\xi),
 \qquad \xi\in\R^d.
\end{aligned}
\]
As $K_{j,t}$, $\Pcal_y$, and $k_j$ belong to
$L^1(\R^d)$, the convolution theorem and the
injectivity of the Fourier transform imply that
\[
 K_{j,t}-K_{j,t}*\Pcal_y=k_j.
\]
Using \eqref{eq:dyadic-convolution} and Fubini's theorem,
we obtain
\[
\begin{aligned}
 \Delta_jg
 &=k_j*g\\
 &=(K_{j,t}-K_{j,t}*\Pcal_y)*g\\
 &=K_{j,t}*(g-\Pcal_y*g).
\end{aligned}
\]
Therefore, Young's inequality and
\eqref{eq:at-kernel-bound} yield
\[
\begin{aligned}
 \|\Delta_jg\|_{L^\infty}
 &\le
 \|K_{j,t}\|_{L^1}
 \|g-\Pcal_y*g\|_{L^\infty}\\
 &=\|\check a_t\|_{L^1}E(y)\\
 &\le C_{d,1}E(y).
\end{aligned}
\]
Multiplying the preceding inequality by $y^{-2}$ and integrating
over $[2^{-j-1},2^{-j}]$, we obtain
\[
 \norm{\Delta_jg}_{L^\infty}
 \int_{2^{-j-1}}^{2^{-j}}\frac{\dd y}{y^2}
 \le
 C_{d,1}\int_{2^{-j-1}}^{2^{-j}}
 \frac{E(y)}{y^2}\dd y.
\]
Since
\[
 \int_{2^{-j-1}}^{2^{-j}}\frac{\dd y}{y^2}
 =
 2^{j+1}-2^j
 =
 2^j,
\]
it follows that
\[
 2^j\norm{\Delta_jg}_{L^\infty}
 \le
 C_{d,1}\int_{2^{-j-1}}^{2^{-j}}
 \frac{E(y)}{y^2}\dd y.
\]
Summing over $j\ge0$ proves
\begin{equation}\label{eq:dyadic-from-poisson}
 \sum_{j\ge0}2^j\norm{\Delta_jg}_{L^\infty}
 \le
 C_{d,1}\int_0^1\frac{E(y)}{y^2}\dd y.
\end{equation}

To prove the reverse inequality, set
\[
 q_d(s)=
 \begin{cases}
  \min\{1,s^2(1+|\log s|)\},&d=2,\\
  \min\{1,s^2\},&d>2.
 \end{cases}
\]

We first estimate $\Pcal_y*\Delta_jg-\Delta_jg$. For $j\ge0$, we
claim that
\begin{equation}\label{eq:dyadicPoissonEstimate}
 \norm{\Pcal_y*\Delta_jg-\Delta_jg}_{L^\infty}
 \le
 C_{d,2}C_{d,3}q_d(2^jy)
 \norm{\Delta_jg}_{L^\infty},
\end{equation}
where $C_{d,2}$ and $C_{d,3}$ are chosen below.
Let $h=\Delta_jg$. For $x\in\R^d$ and $y>0$,
the definition of convolution gives
\[
 (\Pcal_y*h)(x)
 =
 \int_{\R^d}\Pcal_y(z)h(x-z)\dd z.
\]
Since $\Pcal_y$ is even, the change of variables $z\mapsto-z$ gives
\[
 (\Pcal_y*h)(x)
 =
 \int_{\R^d}\Pcal_y(z)h(x+z)\dd z.
\]
Averaging these two identities, we obtain
\[
 (\Pcal_y*h)(x)
 =
 \frac12\int_{\R^d}\Pcal_y(z)
 \bigl(h(x+z)+h(x-z)\bigr)\dd z.
\]
Moreover, since
\[
 \int_{\R^d}\Pcal_y(z)\dd z=1,
\]
we have
\[
 h(x)
 =
 \frac12\int_{\R^d}\Pcal_y(z)\,2h(x)\dd z.
\]
Subtracting the last two identities gives
\begin{equation}\label{eq:poisson-second-difference}
 (\Pcal_y*h)(x)-h(x)
 =
 \frac12\int_{\R^d}\Pcal_y(z)
 \bigl(h(x+z)-2h(x)+h(x-z)\bigr)\dd z.
\end{equation}
Since $h=\Delta_jg$ and
\[
 \supp\widehat h
 \subset
 \{\xi\in\R^d:|\xi|\le2^{j+1}\},
\]
Lemma~\ref{lem:Bernstein}, applied with $R=2^{j+1}$ and
$|\beta|=2$, gives
\[
 \norm{\nabla^2h}_{L^\infty}
 \le
 C_{d,2}2^{2j}\norm{h}_{L^\infty},
\]
where $C_{d,2}\ge4$ depends only on $d$.
For any $x,z\in\mathbb R^d$, applying the second-order Taylor
formula with integral remainder to $t\mapsto h(x+tz)$ and
$t\mapsto h(x-tz)$, we obtain
\[
 |h(x+z)-2h(x)+h(x-z)|
 \le
 |z|^2\norm{\nabla^2h}_{L^\infty}
 \le
 C_{d,2}2^{2j}|z|^2\norm h_{L^\infty}.
\]
On the other hand, the triangle inequality gives
\[
\begin{aligned}
 |h(x+z)-2h(x)+h(x-z)|
 &\le |h(x+z)|+2|h(x)|+|h(x-z)|\\
 &\le4\norm h_{L^\infty}.
\end{aligned}
\]
Combining these two estimates, we obtain
\begin{equation}\label{eq:dyadic-second-difference}
 |h(x+z)-2h(x)+h(x-z)|
 \le
 C_{d,2}\min\{1,2^{2j}|z|^2\}
 \norm{h}_{L^\infty}.
\end{equation}
Since $h=\Delta_jg$ and $\Pcal_y\ge0$, it follows from
\eqref{eq:poisson-second-difference} and
\eqref{eq:dyadic-second-difference} that
\[
 \norm{\Pcal_y*\Delta_jg-\Delta_jg}_{L^\infty}
 \le
 \frac{C_{d,2}}{2}\norm{\Delta_jg}_{L^\infty}
 \int_{\R^d}\Pcal_y(z)
 \min\{1,2^{2j}|z|^2\}\dd z.
\]
By the dilation identity
\[
 \Pcal_y(z)=y^{-d}\Pcal_1(z/y),
\]
the change of variables $z=yw$ gives
\[
\begin{aligned}
 &\int_{\R^d}\Pcal_y(z)
 \min\{1,2^{2j}|z|^2\}\dd z\\
 &\qquad=
 \int_{\R^d}y^{-d}\Pcal_1(z/y)
 \min\{1,2^{2j}|z|^2\}\dd z\\
 &\qquad=
 \int_{\R^d}\Pcal_1(w)
 \min\{1,2^{2j}y^2|w|^2\}\dd w.
\end{aligned}
\]
Putting $s=2^jy$, we conclude that
\[
 \norm{\Pcal_y*\Delta_jg-\Delta_jg}_{L^\infty}
 \le
 C_{d,2}\norm{\Delta_jg}_{L^\infty}
 \int_{\R^d}\Pcal_1(w)
 \min\{1,s^2|w|^2\}\dd w.
\]
Assume first that $0<s\le1$, so that $s^{-1}\ge1$.
From \eqref{eq:kernel} with $y=1$, we have
\[
 \Pcal_1(w)=c_d(1+|w|^2)^{-d},
 \qquad w\in\R^d,
\]
and hence
\[
 \Pcal_1(w)\le c_d
 \quad (|w|\le1),
 \qquad
 \Pcal_1(w)\le c_d|w|^{-2d}
 \quad (|w|>1).
\]
Splitting $\R^d$ into
\[
 \{|w|\le1\},\qquad
 \{1<|w|\le s^{-1}\},\qquad
 \{|w|>s^{-1}\},
\]
and using polar coordinates, we obtain
\[
\begin{aligned}
 &\int_{\R^d}\Pcal_1(w)
 \min\{1,s^2|w|^2\}\dd w\\
 &\quad\le
 c_d|\mathbb S^{d-1}|
 \left(
 s^2\int_0^1r^{d+1}\dd r
 +s^2\int_1^{s^{-1}}r^{1-d}\dd r
 +\int_{s^{-1}}^\infty r^{-d-1}\dd r
 \right)\\
 &\quad=
 c_d|\mathbb S^{d-1}|
 \left(
 \frac{s^2}{d+2}
 +s^2\int_1^{s^{-1}}r^{1-d}\dd r
 +\frac{s^d}{d}
 \right).
\end{aligned}
\]
Here $|\mathbb S^{d-1}|$ denotes the surface area
of the unit sphere in $\R^d$.
If $d=2$, then
\[
 \int_1^{s^{-1}}r^{-1}\dd r
 =
 |\log s|.
\]
If $d>2$, then
\[
 \int_1^{s^{-1}}r^{1-d}\dd r
 \le
 \int_1^\infty r^{1-d}\dd r
 =
 \frac1{d-2}.
\]
Since $s^d\le s^2$ for $0<s\le1$, it follows that
\[
 \int_{\R^d}\Pcal_1(w)
 \min\{1,s^2|w|^2\}\dd w
 \le
 C_{d,3}
 \begin{cases}
  s^2(1+|\log s|),&d=2,\\
  s^2,&d>2,
 \end{cases}
\]
where
$C_{d,3}:=\max\{1,2c_d|\mathbb S^{d-1}|\}$.
For $s\ge1$, we have $q_d(s)=1$, and hence
\[
\begin{aligned}
 \int_{\R^d}\Pcal_1(w)
 \min\{1,s^2|w|^2\}\dd w
 &\le \int_{\R^d}\Pcal_1(w)\dd w\\
 &=1=q_d(s).
\end{aligned}
\]
Combining this with the estimate for $0<s\le1$, we obtain
\begin{equation}\label{eq:poisson-kernel-q-bound}
 \int_{\R^d}\Pcal_1(w)
 \min\{1,s^2|w|^2\}\dd w
 \le C_{d,3}q_d(s),
 \qquad s>0.
\end{equation}
Taking $s=2^jy$ in the preceding estimate gives
\eqref{eq:dyadicPoissonEstimate}.

We next estimate $\bigl(\Pcal_y*S_0g\bigr)(x)-S_0g(x)$. Recall from the proof of Lemma~\ref{lem:DZLip} that
\[
 S_0g=L*g,
 \qquad \widehat L=\varphi.
\]
Since $\varphi$ is smooth and compactly supported, we have
\[
 L,\ \nabla^2L\in L^1(\R^d)
\]
(see \cite[Example~2.2.2, Corollary~2.2.15,
and Proposition~2.2.6]{Grafakos}). Hence
\[
 \nabla^2S_0g=(\nabla^2L)*g,
\]
and Young's inequality gives
\[
\begin{aligned}
 \norm{S_0g}_{L^\infty}
 +
 \norm{\nabla^2S_0g}_{L^\infty}
 &\le
 \left(
 \norm L_{L^1}
 +
 \norm{\nabla^2L}_{L^1}
 \right)\norm g_{L^\infty}\\
 &\le
 C_{d,4}\norm g_{L^\infty},
\end{aligned}
\]
where
\[
 C_{d,4}:=
 4\left(\norm L_{L^1}+\norm{\nabla^2L}_{L^1}\right)
\]
depends only on $d$ for the fixed choice of $\psi$.
For any $x,z\in\R^d$, the second-order Taylor formula gives
\[
\begin{aligned}
 &|S_0g(x+z)-2S_0g(x)+S_0g(x-z)|\\
 &\qquad\le
 |z|^2\norm{\nabla^2S_0g}_{L^\infty}\\
 &\qquad\le
 C_{d,4}|z|^2\norm g_{L^\infty}.
\end{aligned}
\]
On the other hand, the triangle inequality gives
\[
\begin{aligned}
 &|S_0g(x+z)-2S_0g(x)+S_0g(x-z)|\\
 &\qquad\le
 |S_0g(x+z)|+2|S_0g(x)|+|S_0g(x-z)|\\
 &\qquad\le
 4\norm{S_0g}_{L^\infty}\\
 &\qquad\le
 C_{d,4}\norm g_{L^\infty}.
\end{aligned}
\]
Combining these two estimates, we obtain
\[
 |S_0g(x+z)-2S_0g(x)+S_0g(x-z)|
 \le
 C_{d,4}\min\{1,|z|^2\}\norm g_{L^\infty}.
\]
The argument leading to \eqref{eq:poisson-second-difference},
with $h=S_0g$, gives
\[
 \bigl(\Pcal_y*S_0g\bigr)(x)-S_0g(x)
 =
 \frac12\int_{\R^d}\Pcal_y(z)
 \bigl(S_0g(x+z)-2S_0g(x)+S_0g(x-z)\bigr)\dd z.
\]
Using the second-difference estimate for $S_0g$ obtained above
and taking the supremum over $x\in\R^d$, we obtain
\[
\begin{aligned}
 \norm{\Pcal_y*S_0g-S_0g}_{L^\infty}
 &\le
 C_{d,4}\norm g_{L^\infty}
 \int_{\R^d}\Pcal_y(z)\min\{1,|z|^2\}\dd z.
\end{aligned}
\]
For $y>0$, by the change of variables $z=yw$ with
$z,w\in\R^d$ and \eqref{eq:poisson-kernel-q-bound}
with $s=y$, we obtain
\[
\begin{aligned}
 \int_{\R^d}\Pcal_y(z)\min\{1,|z|^2\}\dd z
 &=
 \int_{\R^d}\Pcal_1(w)\min\{1,y^2|w|^2\}\dd w\\
 &\le C_{d,3}q_d(y).
\end{aligned}
\]
Consequently,
\[
 \norm{\Pcal_y*S_0g-S_0g}_{L^\infty}
 \le
 C_{d,3}C_{d,4}q_d(y)\norm g_{L^\infty}.
\]
For $0<y\le1$, this gives
\[
 \norm{\Pcal_y*S_0g-S_0g}_{L^\infty}
 \le
 C_{d,3}C_{d,4}\norm g_{L^\infty}
 \begin{cases}
  y^2(1+|\log y|),&d=2,\\
  y^2,&d>2.
 \end{cases}
\]
Consequently,
\begin{equation}\label{eq:lowPoissonEstimate}
 \int_0^1
 \frac{\norm{\Pcal_y*S_0g-S_0g}_{L^\infty}}{y^2}\dd y
 \le
 2C_{d,3}C_{d,4}\norm g_{L^\infty},
\end{equation}
where, in the case $d=2$, we have used
$$\int_0^1(1+|\log y|)\dd y=2.$$

For each $j\ge0$, the change of variables $s=2^jy$ gives
\begin{align}
 \int_0^1\frac{q_d(2^jy)}{y^2}\dd y
 &=
 2^j\int_0^{2^j}\frac{q_d(s)}{s^2}\dd s
 \notag\\
 &\le
 2^j\left(
 \int_0^1\frac{q_d(s)}{s^2}\dd s
 +
 \int_1^\infty\frac{\dd s}{s^2}
 \right)
 \notag\\
 &\le 3\cdot 2^j.
 \label{eq:qIntegralEstimate}
\end{align}

Indeed,
\[
 \frac{q_d(s)}{s^2}
 =
 \begin{cases}
  1+|\log s|,&d=2,\\
  1,&d>2,
 \end{cases}
 \qquad 0<s\le1,
\]
and both functions are integrable on $(0,1]$.

We may assume that
\[
 \sum_{j\ge0}2^j\norm{\Delta_jg}_{L^\infty}<\infty,
\]
since otherwise the reverse inequality is trivial.
Then the decomposition
\[
 g=S_0g+\sum_{j\ge0}\Delta_jg
\]
converges uniformly on $\R^d$.
Since $\norm{\Pcal_y}_{L^1}=1$, we may convolve this
series term by term. Thus the triangle inequality gives
\[
\begin{aligned}
 E(y)
 &=
 \norm{\Pcal_y*g-g}_{L^\infty}\\
 &\le
 \norm{\Pcal_y*S_0g-S_0g}_{L^\infty}
 +
 \sum_{j\ge0}
 \norm{\Pcal_y*\Delta_jg-\Delta_jg}_{L^\infty}.
\end{aligned}
\]
Using \eqref{eq:dyadicPoissonEstimate}, we obtain
\begin{equation}\label{eq:EdyadicEstimate}
\begin{aligned}
 E(y)
 &\le
 \norm{\Pcal_y*S_0g-S_0g}_{L^\infty}\\
 &\quad+
 C_{d,2}C_{d,3}\sum_{j\ge0}
 q_d(2^jy)\norm{\Delta_jg}_{L^\infty}.
\end{aligned}
\end{equation}
Dividing \eqref{eq:EdyadicEstimate} by $y^2$ and integrating over
$(0,1)$, Tonelli's theorem gives
\[
\begin{aligned}
 \int_0^1\frac{E(y)}{y^2}\dd y
 &\le
 \int_0^1
 \frac{\norm{\Pcal_y*S_0g-S_0g}_{L^\infty}}{y^2}\dd y\\
 &\quad+
 C_{d,2}C_{d,3}\sum_{j\ge0}\norm{\Delta_jg}_{L^\infty}
 \int_0^1\frac{q_d(2^jy)}{y^2}\dd y.
\end{aligned}
\]
Applying \eqref{eq:lowPoissonEstimate} and
\eqref{eq:qIntegralEstimate}, we conclude that
\[
 \int_0^1\frac{E(y)}{y^2}\dd y
 \le
 C_d\left(
 \norm g_{L^\infty}
 +
 \sum_{j\ge0}2^j\norm{\Delta_jg}_{L^\infty}
 \right),
\]
where $C_d:=C_{d,3}\max\{2C_{d,4},3C_{d,2}\}$.
Together with \eqref{eq:BesovDyadic} and
\eqref{eq:dyadic-from-poisson}, this proves the asserted equivalence.
\end{proof}
We now turn to the proofs of the main results.

\section{Proofs of the main results}

We first prove the subcritical criteria using the estimates in
Section~2. We then give the endpoint counterexamples and prove the
results under the Dini--Zygmund condition.

\subsection{The transversal and modulus criteria}

We begin with an auxiliary lemma for the proof of the second
assertion of Theorem~\ref{thm:A}.

\begin{lemma}
\label{lem:sign}
Let $y>0$, $0<\alpha\le1$, and let $v:[0,y]\to\R$
be continuous. Put $a=v(0)$. Suppose that there is a constant
$M\ge0$ such that
\[
  \bigl||v(t)|-|a|\bigr|\le Mt^\alpha,
  \qquad 0\le t\le y.
\]
Then
\[
  |v(y)-a|\le3My^\alpha.
\]
\end{lemma}

\begin{proof}
If $v(y)a\ge0$, then
\[
 |v(y)-a|=\bigl||v(y)|-|a|\bigr|.
\]
If $v(y)a<0$, continuity gives $\tau\in(0,y)$ with $v(\tau)=0$.  Hence
\[
 |a|\le M\tau^\alpha\le My^\alpha
\]
and
\[
 |v(y)|\le|a|+My^\alpha\le2My^\alpha.
\]
Adding the last two bounds proves the claim.
\end{proof}

\begin{proof}[Proof of Theorem \ref{thm:A}]
Put $V=\Vcal_\alpha(u)$ and assume that $V<\infty$. For $s,t>0$,
the definition of $\Vcal_\alpha(u)$ gives
\[
  \norm{u(\cdot,s)-u(\cdot,t)}_{L^\infty}
  \le V|s-t|^\alpha.
\]
It follows that $\{u(\cdot,y)\}_{y>0}$ is uniformly Cauchy on
$\R^d$ as $y\to0^+$. Hence there exists a mapping
$f:\R^d\to\R^m$ such that $u(\cdot,y)\to f$ uniformly.
Letting $t\to0^+$ in the preceding inequality gives
\begin{equation}\label{eq:traceRate}
  \norm{u(\cdot,y)-f}_{L^\infty}
  \le Vy^\alpha.
\end{equation}
Since $u$ is bounded,
\[
  \norm{f}_{L^\infty}
  \le \norm{u}_{L^\infty(\Hh^n)}
  <\infty.
\]
Each $u(\cdot,y)$ is continuous, so its uniform limit $f$ is
continuous. Thus $f$ is a uniform boundary trace of $u$. If $g$ is
another uniform boundary trace, then
\[
  \norm{f-g}_{L^\infty}
  \le
  \norm{f-u(\cdot,y)}_{L^\infty}
  +
  \norm{u(\cdot,y)-g}_{L^\infty}
  \longrightarrow0
\]
as $y\to0^+$. Hence $f=g$, and the uniform boundary trace is unique.

By the Poisson representation in Section~2.2,
\[
  u(\cdot,y)=\Pcal_y*f.
\]
Therefore, \eqref{eq:traceRate} gives
\[
  \sup_{y>0}y^{-\alpha}
  \norm{\Pcal_y*f-f}_{L^\infty}
  =
  \sup_{y>0}y^{-\alpha}
  \norm{u(\cdot,y)-f}_{L^\infty}
  \le V.
\]
Let $A_{d,\alpha}$ be the constant in
Lemma~\ref{lem:inverse}.  Then
\[
  [f]_{\alpha,\R^d}
  \le A_{d,\alpha}V.
\]
Thus $f\in\Lambda_\alpha(\R^d;\R^m)$.
Define $u(x,0)=f(x)$. Let $B_{d,\alpha}$ be the constant in
Lemma~\ref{lem:extension}. By this lemma, $u$ is continuous on
$\overline{\Hh^n}$ and
\[
  [u]_{\alpha,\Hh^n}
  \le B_{d,\alpha}[f]_{\alpha,\R^d}
  \le A_{d,\alpha}B_{d,\alpha}V.
\]

On the other hand, restriction to each vertical line gives
\[
  \Vcal_\alpha(u)\le [u]_{\alpha,\Hh^n}.
\]
Setting
\[
  C_{d,\alpha}
  =A_{d,\alpha}B_{d,\alpha},
\]
we conclude that
\[
  \Vcal_\alpha(u)
  \le [u]_{\alpha,\Hh^n}
  \le C_{d,\alpha}\Vcal_\alpha(u).
\]

We now prove the second assertion. Suppose that $m=1$ and that $u$
extends continuously to $\overline{\Hh^n}$, with boundary trace
$f(x)=u(x,0)$. For $x\in\R^d$ and $y>0$, we have
\[
\begin{aligned}
  \bigl||u(x,y)|-|f(x)|\bigr|
  &\le |u(x,y)-f(x)| \\
  &=|u(x,y)-u(x,0)| \\
  &\le [u]_{\alpha,\Hh^n}y^\alpha.
\end{aligned}
\]
Taking the supremum over $x\in\R^d$ and $y>0$ gives
\[
  \Mcal_\alpha(u)\le [u]_{\alpha,\Hh^n}.
\]

It remains to prove the reverse estimate. Put
$M=\Mcal_\alpha(u)$ and assume that $M<\infty$. For each
$x\in\R^d$, apply Lemma~\ref{lem:sign} to
\[
  v(t)=u(x,t)
  \qquad\text{and}\qquad
  a=f(x).
\]
The resulting estimate is uniform in $x$ and gives
\[
  \norm{u(\cdot,y)-f}_{L^\infty}
  \le 3My^\alpha.
\]
Thus $f$ is a uniform boundary trace of $u$.
By the Poisson representation in Section~2.2, $u=\Pcal[f]$.
Hence, using the notation of Lemma~\ref{lem:inverse},
\[
  A_\alpha(f)
  =
  \sup_{y>0}y^{-\alpha}
  \norm{u(\cdot,y)-f}_{L^\infty}
  \le 3M.
\]
Let $A_{d,\alpha}$ and $B_{d,\alpha}$ denote the constants in
Lemmas~\ref{lem:inverse} and \ref{lem:extension}, respectively.
Since $f\in L^\infty(\R^d)$ is continuous, Lemma~\ref{lem:inverse}
yields
\[
  f\in\Lambda_\alpha(\R^d;\R),
  \qquad
  [f]_{\alpha,\R^d}
  \le 3A_{d,\alpha}M.
\]
Lemma~\ref{lem:extension} then gives
\[
  [u]_{\alpha,\Hh^n}
  \le B_{d,\alpha}[f]_{\alpha,\R^d}
  \le 3A_{d,\alpha}B_{d,\alpha}M.
\]
Setting $C'_{d,\alpha}=3A_{d,\alpha}B_{d,\alpha}$, we conclude that
\[
  \Mcal_\alpha(u)
  \le [u]_{\alpha,\Hh^n}
  \le C'_{d,\alpha}\Mcal_\alpha(u).
\]\end{proof}


We now prove Proposition~\ref{prop:counterexample} by a lacunary
construction.

\subsection{A lacunary counterexample at \texorpdfstring{$\alpha=1$}{alpha=1}}

\begin{proof}[Proof of Proposition \ref{prop:counterexample}]
We first prove part~\ref{item:counterexample-transversal}. For $x\in\R^d$, let $x_1$ denote
the first coordinate of $x$, and define
\begin{equation}\label{eq:lacunary-F}
  F(x)
  =
  \sum_{j=1}^{\infty}2^{-j}\sin(2^jx_1).
\end{equation}
For integers $M>N\ge1$, we have
\[
  \sup_{x\in\R^d}
  \left|
    \sum_{j=N+1}^{M}2^{-j}\sin(2^jx_1)
  \right|
  \le
  \sum_{j=N+1}^{M}2^{-j}
  \le 2^{-N}.
\]
Thus the series converges uniformly on $\R^d$. Let
\[
  F_M(x)=\sum_{j=1}^{M}2^{-j}\sin(2^jx_1).
\]
Each $F_M$ is continuous, and $F_M\to F$ uniformly on $\R^d$.
Therefore, $F$ is continuous. Moreover,
\[
  |F(x)|
  \le \sum_{j=1}^{\infty}2^{-j}
  =1,
\]
so $F$ is bounded. Hence $F$ is bounded and continuous.

We show that it is not Lipschitz. Put $h_N=2^{-N}$. Since
$F(0)=0$, we have
\begin{equation}\label{eq:F-dyadic-increment}
F(h_Ne_1)-F(0)
=\sum_{j=1}^{\infty}2^{-j}\sin(2^{j-N}).
\end{equation}
If
\[
  c_0
  =
  \min_{0<t\le1}\frac{\sin t}{t}
  >0,
\]
then $\sin t\ge c_0t$ for $0<t\le1$. For $1\le j\le N$, we have
$0<2^{j-N}\le1$, and hence
\[
\begin{aligned}
  2^{-j}\sin(2^{j-N})
  &\ge c_0\,2^{-j}2^{j-N}\\
  &=c_0\,2^{-N}.
\end{aligned}
\]
Summing over $1\le j\le N$, we obtain
\begin{equation}\label{eq:firstNlower}
  \sum_{j=1}^{N}2^{-j}\sin(2^{j-N})
  \ge c_0N2^{-N}.
\end{equation}
For the remaining terms, put
\[
  R_N
  =
  \sum_{j>N}2^{-j}\sin(2^{j-N}).
\]
Since $|\sin t|\le1$, we have
\begin{equation}\label{eq:tailRN}
  |R_N|
  \le
  \sum_{j>N}2^{-j}
  =
  2^{-N}.
\end{equation}
By \eqref{eq:F-dyadic-increment}--\eqref{eq:tailRN}, we obtain
\[
\begin{aligned}
  |F(h_Ne_1)-F(0)|
  &\ge
  \sum_{j=1}^{N}2^{-j}\sin(2^{j-N})-|R_N|\\
  &\ge
  (c_0N-1)2^{-N}.
\end{aligned}
\]
Since $h_N=2^{-N}\to0$ as $N\to\infty$, we have
\begin{equation}\label{eq:FnotLipschitz}
  \frac{|F(h_Ne_1)-F(0)|}{h_N}
  \ge c_0N-1
  \longrightarrow\infty
  \qquad\text{as }N\to\infty.
\end{equation}
Thus $F$ is not Lipschitz continuous.

We now take $F$ as boundary data and construct the function asserted
in part~\ref{item:counterexample-transversal}. For $y\ge0$, define
\begin{equation}\label{eq:lacunary-U}
  U(x,y)
  =
  \sum_{j=1}^{\infty}
  2^{-j}m_d(2^jy)\sin(2^jx_1).
\end{equation}
Since $0<m_d(t)\le1$ for $t\ge0$, for integers $M>N\ge1$ we have
\[
  \sup_{(x,y)\in\overline{\Hh^n}}
  \left|
    \sum_{j=N+1}^{M}
    2^{-j}m_d(2^jy)\sin(2^jx_1)
  \right|
  \le
  \sum_{j=N+1}^{M}2^{-j}
  \le2^{-N}.
\]
Thus the series converges uniformly on
$\overline{\Hh^n}$, and its sum $U$ is continuous there. Moreover,
\begin{equation}\label{eq:Ubound}
  |U(x,y)|
  \le
  \sum_{j=1}^{\infty}2^{-j}
  =
  1,
  \qquad (x,y)\in\overline{\Hh^n},
\end{equation}
and hence $U$ is bounded. Since $m_d(0)=1$, we also have
\[
  U(x,0)
  =
  \sum_{j=1}^{\infty}2^{-j}\sin(2^jx_1)
  =
  F(x).
\]

Let
\[
  F_N(x)
  =
  \sum_{j=1}^{N}2^{-j}\sin(2^jx_1).
\]
For $\xi\in\R^d$, put
\[
  E_\xi(x)=e^{i\xi\cdot x},
  \qquad x\in\R^d.
\]
For $x\in\R^d$ and $y>0$, the definition of convolution and
\eqref{eq:multiplier} give
\[
\begin{aligned}
  (\Pcal_y*E_\xi)(x)
  &=
  \int_{\R^d}\Pcal_y(z)E_\xi(x-z)\dd z\\
  &=
  \int_{\R^d}\Pcal_y(z)e^{i\xi\cdot(x-z)}\dd z\\
  &=
  e^{i\xi\cdot x}
  \int_{\R^d}\Pcal_y(z)e^{-i\xi\cdot z}\dd z\\
  &=
  e^{i\xi\cdot x}\widehat{\Pcal_y}(\xi)\\
  &=
  m_d(y|\xi|)e^{i\xi\cdot x}.
\end{aligned}
\]
Let $e_1=(1,0,\ldots,0)\in\R^d$. For integers $j\ge1$, put
\[
  \phi_j(x)
  =\operatorname{Im}E_{2^je_1}(x)
  =\operatorname{Im}e^{i2^jx_1}
  =\sin(2^jx_1),
  \qquad x\in\R^d.
\]
Since $\Pcal_y$ and $m_d$ are real-valued, taking $\xi=2^je_1$
and then taking imaginary parts in the preceding convolution
identity, we obtain
\[
  (\Pcal_y*\phi_j)(x)
  =
  m_d(2^jy)\sin(2^jx_1).
\]
Since
\[
  F_N=\sum_{j=1}^{N}2^{-j}\phi_j,
\]
we have
\[
\begin{aligned}
  \Pcal_y*F_N(x)
  &=
  \sum_{j=1}^{N}2^{-j}(\Pcal_y*\phi_j)(x)\\
  &=
  \sum_{j=1}^{N}
  2^{-j}m_d(2^jy)\sin(2^jx_1).
\end{aligned}
\]
For every $x\in\R^d$, we have
\[
\begin{aligned}
  |\Pcal_y*F_N(x)-\Pcal_y*F(x)|
  &=
  \left|
    \int_{\R^d}\Pcal_y(z)
    \bigl(F_N(x-z)-F(x-z)\bigr)\dd z
  \right|\\
  &\le
  \norm{F_N-F}_{L^\infty}
  \int_{\R^d}\Pcal_y(z)\dd z\\
  &=
  \norm{F_N-F}_{L^\infty}.
\end{aligned}
\]
Taking the supremum over $x\in\R^d$, we obtain
\[
  \norm{\Pcal_y*F_N-\Pcal_y*F}_{L^\infty}
  \le
  \norm{F_N-F}_{L^\infty}
  \longrightarrow0
  \qquad\text{as }N\to\infty,
\]
where the last limit follows from the uniform convergence of $F_N$
to $F$. Letting $N\to\infty$, we obtain
\[
\begin{aligned}
  \Pcal_y*F
  &=
  \lim_{N\to\infty}\Pcal_y*F_N\\
  &=
  \lim_{N\to\infty}
  \sum_{j=1}^{N}
  2^{-j}m_d(2^jy)\sin(2^jx_1)\\
  &=
  U(\cdot,y).
\end{aligned}
\]
Hence $U=\Pcal[F]$. Since $F\in L^\infty(\R^d)$, Section~2.2 gives
\begin{equation}\label{eq:U-hyperbolic-harmonic}
  \Delta_hU=0
  \qquad\text{in }\Hh^n.
\end{equation}
By Lemma~\ref{lem:Bessel}\ref{item:Bessel-pointwise}, the differentiated
series converges uniformly on compact subsets of $\Hh^n$.
Hence termwise differentiation in $y$ gives
\[
  \partial_yU(x,y)
  =
  \sum_{j=1}^{\infty}
  m_d'(2^jy)\sin(2^jx_1),
  \qquad x\in\R^d,\quad y>0.
\]
By \eqref{eq:dyadicBessel}, there is a constant $C_d>0$ such that
\[
  \sup_{(x,y)\in\Hh^n}
  |\partial_yU(x,y)|
  \le C_d.
\]
By the mean value theorem and the continuity of $U$ on
$\overline{\Hh^n}$, we obtain
\[
  |U(x,s)-U(x,t)|
  \le C_d|s-t|,
  \qquad x\in\R^d,\quad s,t\geq0.
\]
If $U$ were Lipschitz continuous on $\overline{\Hh^n}$, then its
boundary restriction $F=U(\cdot,0)$ would also be Lipschitz,
contradicting \eqref{eq:FnotLipschitz}. Hence $U$ is not Lipschitz
continuous on $\overline{\Hh^n}$. This proves part~\ref{item:counterexample-transversal}.

We next prove part~\ref{item:counterexample-modulus}.
For $x\in\R^d$ and $y>0$,
\eqref{eq:lacunary-F} and \eqref{eq:lacunary-U} give
\[
  U(x,y)-F(x)
  =
  \sum_{j=1}^{\infty}
  2^{-j}\bigl(m_d(2^jy)-1\bigr)\sin(2^jx_1).
\]
Since $|\sin(2^jx_1)|\le1$, we have
\[
  |U(x,y)-F(x)|
  \le
  \sum_{j=1}^{\infty}
  2^{-j}|1-m_d(2^jy)|.
\]
Taking the supremum over $x\in\R^d$ and applying
\eqref{eq:dyadicApprox} with $r=y$, we obtain
\begin{equation}\label{eq:UFapprox}
  \norm{U(\cdot,y)-F}_{L^\infty(\R^d)}
  \le
  \sum_{j=1}^{\infty}
  2^{-j}|1-m_d(2^jy)|
  \le C'_dy,
  \qquad y>0,
\end{equation}
where $C'_d=C_d'''$ and $C_d'''$ is the constant in
\eqref{eq:dyadicApprox}.
Choose a real constant
\[
  A>1+\norm{F}_{L^\infty(\R^d)}
\]
and, for $x\in\R^d$ and $y\ge0$, define
\[
  V(x,y)=A+U(x,y),
  \qquad
  G(x)=A+F(x).
\]
Since $\Delta_hA=0$, \eqref{eq:U-hyperbolic-harmonic}
and the linearity of $\Delta_h$ give
\[
  \Delta_hV=0
  \qquad\text{in }\Hh^n.
\]
Moreover, $V$ is bounded and continuous on
$\overline{\Hh^n}$. Since $U(x,0)=F(x)$,
\[
  V(x,0)=A+U(x,0)=A+F(x)=G(x).
\]
By \eqref{eq:UFapprox},
\[
  \norm{V(\cdot,y)-G}_{L^\infty}
  =
  \norm{U(\cdot,y)-F}_{L^\infty}
  \le C'_dy
  \longrightarrow0
  \qquad\text{as }y\to0^+.
\]
Thus $G$ is the uniform boundary trace of $V$.

By \eqref{eq:Ubound} and the choice of $A$,
\[
  V(x,y)\ge A-|U(x,y)|\ge A-1>0
\]
and
\[
  G(x)\ge A-|F(x)|
  \ge A-\norm{F}_{L^\infty}>0.
\]
Hence both $V$ and $G$ are positive. Consequently, \eqref{eq:UFapprox} gives
\[
  \bigl||V(x,y)|-|G(x)|\bigr|
  =
  |U(x,y)-F(x)|
  \le C'_dy.
\]
For $x,x'\in\R^d$, we have
\[
  |G(x)-G(x')|
  =
  |F(x)-F(x')|.
\]
Hence \eqref{eq:FnotLipschitz} shows that $G$ is not Lipschitz
continuous. If $V$ were Lipschitz continuous on
$\overline{\Hh^n}$, then its boundary restriction
$G=V(\cdot,0)$ would be Lipschitz, which is a contradiction.
Therefore, $V$ is not Lipschitz continuous on
$\overline{\Hh^n}$. This proves
part~\ref{item:counterexample-modulus}.
\end{proof}

\subsection{Endpoint recovery under Dini--Zygmund summability}

\begin{proof}[Proof of Theorem~\ref{thm:endpoint}]
We first prove part~\ref{item:endpoint-transversal}. Put
$
  V=\Vcal_1(u)<\infty.
$
By the definition of $\Vcal_1(u)$, for
$x\in\R^d$ and $s,t>0$,
\[
  |u(x,s)-u(x,t)|
  \le V|s-t|.
\]
Taking the supremum over $x\in\R^d$, we obtain
\[
  \norm{u(\cdot,s)-u(\cdot,t)}_{L^\infty}
  \le V|s-t|.
\]
It follows that $\{u(\cdot,y)\}_{y>0}$ is uniformly Cauchy on
$\R^d$ as $y\to0^+$. Hence there exists a mapping
$f:\R^d\to\R^m$ such that
\[
  u(\cdot,y)\longrightarrow f
  \qquad\text{uniformly as }y\to0^+.
\]
Since $u$ is bounded and each $u(\cdot,y)$ is continuous, $f$ is
bounded and continuous. Letting $t\to0^+$ in
\[
  \norm{u(\cdot,y)-u(\cdot,t)}_{L^\infty}
  \le V|y-t|,
\]
we obtain
\[
  \norm{u(\cdot,y)-f}_{L^\infty}
  \le Vy.
\]
Thus $f$ is a uniform boundary trace of $u$. If $g$ is another
uniform boundary trace, then
\[
  \norm{f-g}_{L^\infty}
  \le
  \norm{f-u(\cdot,y)}_{L^\infty}
  +
  \norm{u(\cdot,y)-g}_{L^\infty}
  \longrightarrow0
\]
as $y\to0^+$. Hence $f=g$.
By the Poisson representation in Section~2.2,
$u=\Pcal[f]$, and therefore
\[
  E_u(y)=\norm{\Pcal_y*f-f}_{L^\infty}.
\]
Let $A_d$, $B_d$ and $K_d$ denote the constants in
Lemmas~\ref{lem:approxBesov}, \ref{lem:DZLip} and
\ref{lem:extension}, respectively, with $\alpha=1$ in the
last lemma.
Since $f$ is bounded and continuous,
Lemma~\ref{lem:approxBesov} and \eqref{eq:intrinsicA} give
\[
  \norm{f}_{L^\infty}+\Dcal(f)
  \le
  A_d\left(
    \norm{f}_{L^\infty}
    +\int_0^1\frac{E_u(t)}{t^2}\dd t
  \right)
  <\infty.
\]
Thus $f\in B_{\infty,1}^{1}(\R^d;\R^m)$, and
Lemma~\ref{lem:DZLip} gives
\[
  [f]_{\Lip(\R^d)}
  \le
  B_d\bigl(\norm{f}_{L^\infty}+\Dcal(f)\bigr)
  <\infty.
\]
In particular,
$f\in L^\infty(\R^d;\R^m)\cap\Lambda_1(\R^d;\R^m)$.
Define $u(x,0)=f(x)$ for $x\in\R^d$.
Applying Lemma~\ref{lem:extension} with $\alpha=1$, we obtain
\[
\begin{aligned}
  [u]_{\Lip(\overline{\Hh^n})}
  &\le K_d[f]_{\Lip(\R^d)}\\
  &\le A_dB_dK_d\left(
    \norm{f}_{L^\infty}
    +\int_0^1\frac{E_u(t)}{t^2}\dd t
  \right).
\end{aligned}
\]
Setting $C_d=A_dB_dK_d$, we obtain
part~\ref{item:endpoint-transversal}.

We next prove part~\ref{item:endpoint-modulus}. Fix $x\in\R^d$ and
$0<y\le1$. If $u(x,y)f(x)\ge0$, then
\[
  |u(x,y)-f(x)|
  =
  \bigl||u(x,y)|-|f(x)|\bigr|
  \le E_{|u|}^{*}(y).
\]

Suppose that $u(x,y)f(x)<0$. Since $f(x)=u(x,0)$ and
$u(x,\cdot)$ is continuous on $[0,y]$, there exists
$\tau\in(0,y)$ such that $u(x,\tau)=0$. By the definition of
$E_{|u|}^{*}(y)$, for every $0<s\le y$,
\[
  \bigl||u(x,s)|-|f(x)|\bigr|
  \le
  \bigl\||u(\cdot,s)|-|f|\bigr\|_{L^\infty}
  \le E_{|u|}^{*}(y).
\]
Taking $s=\tau$ gives
\[
  |f(x)|\le E_{|u|}^{*}(y),
\]
while taking $s=y$ gives
\[
  \bigl||u(x,y)|-|f(x)|\bigr|
  \le E_{|u|}^{*}(y).
\]
Therefore,
\[
\begin{aligned}
  |u(x,y)|
  &\le
  |f(x)|
  +
  \bigl||u(x,y)|-|f(x)|\bigr|\\
  &\le
  |f(x)|+E_{|u|}^{*}(y)\\
  &\le
  2E_{|u|}^{*}(y).
\end{aligned}
\]
Since $u(x,y)$ and $f(x)$ have opposite signs in the second case,
\[
  |u(x,y)-f(x)|
  =
  |u(x,y)|+|f(x)|
  \le 3E_{|u|}^{*}(y).
\]
Thus, in both cases,
\begin{equation}\label{eq:endpointSign}
  \norm{u(\cdot,y)-f}_{L^\infty}
  \le 3E_{|u|}^{*}(y),
  \qquad 0<y\le1.
\end{equation}

It follows from \eqref{eq:endpointSign} and
\eqref{eq:intrinsicB} that $f$ is a uniform boundary trace of $u$.
By the Poisson representation in Section~2.2, $u=\Pcal[f]$.
Using Lemmas~\ref{lem:approxBesov} and \ref{lem:DZLip},
and the case $\alpha=1$ of Lemma~\ref{lem:extension},
the same argument as in the proof of
part~\ref{item:endpoint-transversal} gives
part~\ref{item:endpoint-modulus}.
This completes the proof.
\end{proof}
The real-valued case of Corollary~\ref{cor:boundaryDZ}
uses the following elementary fact. Dyakonov's argument
in \cite[Lemma~3.1]{Dyakonov2006Addendum}, with $\alpha=1$,
applies equally to $\R^d$. We include a proof for completeness.
\begin{lemma}\label{lem:modulusLip}
Let $L\ge0$ and let $g:\R^d\to\R$ be continuous.
If $|g|$ is $L$-Lipschitz, that is,
\[
  \bigl||g(x)|-|g(y)|\bigr|
  \le L|x-y|,
  \qquad x,y\in\R^d,
\]
then $g$ is $L$-Lipschitz.
\end{lemma}

\begin{proof}
If $g(x)g(y)\ge0$, then
\[
  |g(x)-g(y)|
  =
  \bigl||g(x)|-|g(y)|\bigr|
  \le L|x-y|.
\]  If $g(x)g(y)<0$, the continuity of $g$
gives a point $z$ on the segment joining $x$ and $y$ such that
$g(z)=0$. Since $|g|$ is $L$-Lipschitz,
\[
\begin{aligned}
  |g(x)-g(y)|
  &=
  |g(x)|+|g(y)|\\
  &=
  \bigl||g(x)|-|g(z)|\bigr|
  +
  \bigl||g(y)|-|g(z)|\bigr|\\
  &\le
  L|x-z|+L|y-z|\\
  &=
  L|x-y|,
\end{aligned}
\]
where the last equality follows from the fact that
$z$ lies on the line segment joining $x$ and $y$.
\end{proof}
\begin{proof}[Proof of Corollary~\ref{cor:boundaryDZ}]
Since $f$ is a uniform boundary trace of $u$, it is bounded
and continuous. By the representation in Section~2.2,
$u=\Pcal[f]$.

For part~\ref{item:boundary-DZ}, $\Dcal(f)<\infty$ gives
$f\in B_{\infty,1}^{1}(\R^d;\R^m)$ by the definition in
Section~2.3. Lemma~\ref{lem:DZLip} implies that $f$ is
Lipschitz continuous. Applying Lemma~\ref{lem:extension}
with $\alpha=1$, we obtain
$u\in\Lip(\overline{\Hh^n};\R^m)$.

For part~\ref{item:boundary-modulus}, similarly,
$|f|\in B_{\infty,1}^{1}(\R^d;\R)$.
By Lemma~\ref{lem:DZLip}, $|f|$ is Lipschitz continuous.
Since $f$ is real-valued and continuous,
Lemma~\ref{lem:modulusLip} implies that $f$ is Lipschitz
continuous. Lemma~\ref{lem:extension} with $\alpha=1$
then gives $u\in\Lip(\overline{\Hh^n})$.
This completes the proof.
\end{proof}
\section*{Acknowledgments}
We thank Konstantin M.~Dyakonov for bringing the relevant
reference to our attention. This work was supported by the National Natural Science Foundation of China, grant 12271189.

\end{document}